\documentclass[11pt]{amsart}
\usepackage{txfonts}
\usepackage{amsthm,amsmath,amssymb,dsfont,mathrsfs,color}

\newtheorem{thm}{Theorem}[section]
\newtheorem{cor}[thm]{Corollary}
\newtheorem{lem}[thm]{Lemma}
\newtheorem{prop}[thm]{Proposition}

\newtheorem{assu}[thm]{Assumption}
\theoremstyle{definition}
\newtheorem{de}[thm]{Definition}
\theoremstyle{remark}

\numberwithin{equation}{section}

\allowdisplaybreaks
\makeatletter

\newcommand{\Rmnum}[1]{\expandafter\@slowromancap\romannumeral #1@}
\makeatother

\begin{document}

\title{Polynomial Mixing for a Weakly Damped Stochastic Nonlinear Schr\"{o}dinger Equation}

\author{Jing Guo}
\address{J. Guo: School of Mathematics,
Dalian University of Technology, Dalian 116024, P. R. China}
 \email{guojing062@mail.dlut.edu.cn; jingguo062@hotmail.com}

\author{Zhenxin Liu}
\address{Z. Liu (Corresponding author): School of Mathematics,
Dalian University of Technology, Dalian 116024, P. R. China}
\email{zxliu@dlut.edu.cn}

\date{March 31, 2023}



\begin{abstract}
This paper is devoted to proving the polynomial mixing for a weakly damped stochastic nonlinear Schr\"{o}dinger equation with additive noise on a 1D bounded domain. The noise is white in time and smooth in space. We consider both focusing and defocusing nonlinearities, respectively, with exponents of the nonlinearity $\sigma\in[0,2)$  and $\sigma\in[0,\infty)$ and prove the polynomial mixing which implies the uniqueness of the invariant measure by using a coupling method.
\end{abstract}

\maketitle

\section{Introduction}
\setcounter{equation}{0}

The nonlinear Schr\"{o}dinger equation (NLS) is one of the basic nonlinear partial differential equations, which models the propagation of dispersive nonlinear waves. It arises in various areas of physics such as hydrodynamics, optics and plasma physics. Given that randomness and damping has to be taken into account in some circumstances, we need to consider the damped stochastic nonlinear Schr\"{o}dinger equation (SNLS). It is valid for describing waves in long propagation distances. It has the following form
\begin{equation}\label{a1}
\mathrm{d}u(t)= (i\Delta u(t)+i\lambda|u(t)|^{2\sigma}u(t)-\alpha u(t))\mathrm{d}t+b\mathrm{d}W(t),
\end{equation}
where $\alpha>0$, $x\in [0,1]$, $\lambda\in\{1, -1\}$, $u(x,t)$ is a complex-valued unknown function. $\lambda=1$ corresponds to the focusing case and $\lambda=-1$ corresponds to the defocusing case.

It is not difficult to see that if the well-posedness of the NLS has been proved, then for the equation with damping, it can also be easily derived. Therefore, we only recall the relative results on the well-posedness of the NLS. For the deterministic equation, it is well-known that all the solutions exist globally in the subcritical ($\sigma d<2$) case. The first proof of this subject was given by J. Ginibre and G. Velo \cite{GV}. There are also many references, see for example \cite{HNT, KT, W} and references therein. For the stochastic equation, the well-posedness is more difficult to get. A. de. Bouard and A. Debussche \cite{BD1999} demonstrated the local and global existence and uniqueness of square integrable solutions to the focusing NLS with linear multiplicative noise in $\mathds{R}^{n}$ by using the fixed point theorem. They considered the subcritical nonlinearities, where the critical exponent is the same as that of the deterministic equation in dimension 1 or 2, but more restrictive if $n\geq3$. And they investigated the existence and uniqueness of solutions to the NLS with multiplicative or additive noise in $H^{1}(\mathds{R}^{n})$ in \cite{BD2003}. Similarly, the result is more restrictive than that of the deterministic equation. After that by making use of the rescaling approach, V. Barbu et al. proved well-posedness results for the NLS with linear multiplicative noise in $L^{2}(\mathds{R}^{n})$ in the conservative case and nonconservative case in \cite{BRZ2014}. They obtained the result for the subcritical equation with exponents of nonlinearity in the same range as the deterministic case, which improved the results of \cite{BD1999} in the conservative case. And in \cite{BRZ2016}, they discussed the well-posedness of the equation with linear multiplicative noise in $H^{1}(\mathds{R}^{n})$ in the conservative and non-conservative case and considered focusing and defocusing nonlinearities whose exponents are in the same range as the deterministic case, which improved the results of \cite{BD2003} in the special conservative case. For the nonlinear noise, F. Hornung presented the local existence and uniqueness of a solution to the SNLS with subcritical and critical nonlinearities and the global existence and uniqueness of the solution to the equation in the subcritical case under an additional assumption on the nonlinear noise in \cite{H2018}. Besides, in \cite{BM}, Z. Brze\'{z}niak and A. Millet proved the existence and uniqueness of a solution to the SNLS on a two-dimensional compact Riemannian manifold by using stochastic Strichartz estimates. In fact, all the above-mentioned references are about mild solutions. For martingale solutions, we refer to \cite{BHL, BHM, H2020}. And for variational solutions, see \cite{GL, KL2015, KL2016}.

Moreover, as for the existence of an invariant measure for the damped SNLS, there are a few studies. I. Ekren et al. proved the existence of an invariant measure for the NLS with additive noise in $H^{1}(\mathds{R}^{n})$ and the existence of an ergodic measure in \cite{EKM}. J. U. Kim also studied the existence of an invariant measure for the damped SNLS in \cite{KJ}. Z. Brze\'{z}niak et al. considered this equation with defocusing nonlinearity on a 2D compact Riemannian manifold and proved the existence of an invariant measure in \cite{BFZ1}. Besides, they provided some remarks on the uniqueness of the invariant measure in a particular case. But as far as we know, only few studies have shown the uniqueness of an invariant measure for the equation. By using a coupling method, A. Debussche and C. Odasso \cite{DO} proved the uniqueness of an invariant measure for the equation with cubic nonlinearities on a 1D bounded domain. They also revealed that the mixing property holds and that the rate of convergence is at least polynomial of any power. Recently, in \cite{BFZ2}, Z. Brze\'{z}niak et al. proved the uniqueness of the invariant measure for the equation when the damping coefficient is sufficiently large in $\mathds{R}^{n}$ with $n=2$ or $n=3$. In this paper, we focus on proving the uniqueness of an invariant measure for the focusing and defocusing damped SNLS, respectively, with exponents of the nonlinearity $\sigma\in[0,2)$ and $\sigma\in[0,\infty)$ on a 1D bounded domain. In particular, our work generalizes the earlier result of \cite{DO}, where $\sigma=1$ and $\lambda=1$.

In this work, we will use a coupling method to prove the polynomial mixing which implies the uniqueness of the invariant measure. We remark that the result about the existence of an invariant measure can also be obtained by the Krylov-Bogolyubov theorem. To be specific, due to the domain we consider is bounded, we can use some compactness theorem. Besides, we will also need some extra estimates about the solution. Moreover, as far as we know, there are mainly two kinds of methods to prove the uniqueness of the invariant measure. The first one is the Doob theorem or the general Doob theorem. We refer to \cite{DZ1996, HM}. The second one is the coupling method. Due to the lack of smoothing effect in the NLS, we will use the coupling method which can also be used to obtain the rate of convergence, and we will restrict to the case, where only a finite number of modes are forced. But for a more degenerate noise, we cannot deal with it now.

In this paper, as in \cite{DO}, we extensively use the decomposition of $u$ into its low- and high-frequency parts. We assume that the low-frequency part is non-degenerate, but the high-frequency part may be degenerate. We can prove when the low-frequency parts of two solutions from two different initial data are equal, their high-frequency parts will be close, which is the so-called Foias-Prodi estimate. Since the damped SNLS is weakly dissipative, we cannot prove a path-wise Foias-Prodi estimate. We can only prove that it holds on average. Moreover, we are unable to prove the exponential estimate of the growth of solutions in our case. Since the Lyapunov structure is more complicated here, we can only prove the polynomial estimate of the growth of solutions. Therefore, we can only prove that convergence to equilibrium holds with the polynomial speed at any power.

The paper is organized as follows. In sect.2, we present some notations, assumptions and main results. Sect.3 gives some useful priori estimates which will be used to prove the main results and proves the uniqueness of the invariant measure.

\section{Preliminaries and main results}
\setcounter{equation}{0}
\subsection{Notations and assumptions}

We set $A=-\Delta$, $D(A)=H_{0}^{1}([0,1])\bigcap H^{2}([0,1])$. The damped SNLS with the initial data $u_{0}$ and Dirichlet boundary conditions can be written in the form:
\begin{equation}\label{a2}
\begin{cases}
du(t)=-iAu(t)dt+i\lambda|u(t)|^{2\sigma}u(t)dt-\alpha u(t)dt+bdW(t), \quad t\geq0 \\
u(0)=u_{0}\in H_{0}^{1}([0,1]).
\end{cases}
\end{equation}
Let $\left(\Omega,\mathcal{F},\{\mathcal{F}_{t}\}_{t\geq 0},\mathbb{P}\right)$ be a filtered probability space where $\left\{\mathcal{F}_{t}\right\}_{t\geq 0}$ is a filtration satisfying the usual conditions, i.e. $\left\{\mathcal{F}_{t}\right\}_{t\geq 0}$ is right continuous and complete. We recall that a stochastic process $X(t)$ is called non-anticipating with respect to $\left\{\mathcal{F}_{t}\right\}_{t\geq 0}$ if the function $(t,\omega)\mapsto X(t,\omega)$ is measurable, and for each $t\geq0$, $X(t)$ is adapted to $\mathcal{F}(t)$, i.e. $X(t)$ is $\mathcal{F}(t)$-measurable. We denote by $b$ a linear operator on $L^{2}([0,1])$. Let $W$  be a cylindrical Wiener process on $L^{2}([0,1])$. We denote by $\{\mu_{n}\}_{n\in \mathds{N}}$ the increasing sequence of eigenvalues of $A$  and by $\{e_{n}\}_{n\in \mathds{N}}$ the associated eigenvectors. Also, Let $P_{N}$ and $Q_{N}$ be the eigenprojectors onto the space $span\ \{\{e_{k}\}_{0\leq k\leq N}\}$ and onto its complementary space, respectively. We use the Lebesgue space of complex valued functions $L^{p}([0,1])$ endowed with the norm $|\cdot|_{p}$, and the inner
product in $L^{2}([0,1])$ is denoted by $(u,v)=\mathcal{R}\int_{0}^{1}u(x)\bar{v}(x)dx$ for any $u$, $v\in L^{2}([0,1])$, where $\bar{v}$ is the conjugate of $v$ and $\mathcal{R}u$ is the real part of $u$. Let $H^{s}([0,1])$ be the Sobolev space endowed with the norm $\|\cdot\|_{s}$.
For $s\geq 0$, it is not difficult to see that $D\left(A^{\frac{s}{2}}\right)$ is a closed subspace of $H^{s}\left([0,1]\right)$ and $\left\|\cdot\right\|_{s}=\left|A^{\frac{s}{2}}\cdot\right|_{2}$ is equivalent to the usual $H^{s}([0,1])$ norm on this space. Moreover, for any $u\in H^{s}([0,1])$,
\begin{center}
 $D(A^{\frac{s}{2}})=\left \{u=\sum\limits_{k\in \mathds{N}}(u,e_{k})e_{k}\in L^{2}([0,1]) \ \bigg| \ \sum\limits_{k\in \mathds{N}}\mu_{k}^{s}(u,e_{k})^{2}<\infty \right \}$ \indent \ and \  \indent $\|u\|_{s}^{2}=\sum\limits_{k\in \mathds{N}}\mu_{k}^{s}(u,e_{k})^{2}$.
\end{center}

We work under the following assumptions on the noise and the nonlinearity.
\begin{assu}\label{s2}
We suppose that $b$ commutes with $A$, i.e. suppose that $b$ is diagonal in the basis $\{e_{n}\}_{n\in \mathds{N}}$, and we write
$be_{n}=b_{n}e_{n}$, where $b_{n}=(be_{n},e_{n})$.
Moreover, we assume that there exists $N_{*}>0$ such that
$b_{n}>0$ for any $n\leq N_{*}$.
\end{assu}
For any $s\in[0,3]$, we denote by $\mathcal{L}_{2}(L^{2}([0,1]), D(A^{\frac{s}{2}}))$ the space of Hilbert-Schmidt operators from $L^{2}([0,1]$ to $D(A^{\frac{s}{2}})$. Let $b\in\mathcal{L}_{2}(L^{2}([0,1]), D(A^{\frac{3}{2}}))$. We set
$B_{s}:=|b|_{\mathcal{L}_{2}(L^{2}([0,1]),D(A^{\frac{s}{2}}))}^{2}=\sum_{n=0}^{\infty}\mu_{n}^{s}b_{n}^{2}$ for any $s\in[0,3]$.
\begin{assu}\label{s3}
\indent \indent If $\lambda=1$, then $\sigma\in[0,2)$. \\
\indent \indent If $\lambda=-1$, then $\sigma\in[0,\infty)$.
\end{assu}

We use $H_{*}(u)$ to denote the energy, where
$H_{*}(u)=\frac{1}{2}|\nabla u|_{2}^{2}-\frac{\lambda}{2\sigma+2}|u|_{2\sigma+2}^{2\sigma+2}$. It is not difficult to see that when $\lambda=-1$, it is greater than or equal to zero. When $\lambda=1$, it may be negative. But we can modify it by adding a term and recover its nonnegative property. We also denote by $H(u):=H_{*}(u)$ the energy in the former case. And we denote by $H(u):=H_{*}(u)+G|u|_{2}^{2+\frac{4\sigma}{2-\sigma}}$ the modified energy in the latter case, where $G$ is a constant satisfying the inequality
\begin{align}\label{s1}
|u|_{2\sigma+2}^{2\sigma+2}\leq \frac{1}{2\sigma+2}|\nabla u|_{2}^{2}+\frac{G}{2}|u|_{2}^{2+\frac{4\sigma}{2-\sigma}}.
\end{align}
The existence of $G$ can be guaranteed by Gagliardo-Nirenberg's inequality and Young's inequality. When $\lambda=1$, we get
\begin{align}\label{s4}
H(u)&=\frac{1}{2}|\nabla u|_{2}^{2}-\frac{1}{2\sigma+2}|u|_{2\sigma+2}^{2\sigma+2}+G|u|_{2}^{2+\frac{4\sigma}{2-\sigma}} \\ \nonumber
&\geq \frac{2\sigma(\sigma+2)}{(2\sigma+2)^{2}}|\nabla u|_{2}^{2}+\frac{1}{2\sigma+2}|u|_{2\sigma+2}^{2\sigma+2}+\frac{2\sigma+1}{2\sigma+2}G|u|_{2}^{2+\frac{4\sigma}{2-\sigma}}.
\end{align}

We also use the following quantities. We define
$E_{u,k}(t,s)=H^{k}(u(t))+\frac{1}{2}\alpha k\int_{s}^{t}H^{k}(u(r))dr$, $t\geq s$. When $s=0$, we simply write $E_{u,k}(t)$. When $\lambda=1$, we define for any $(u_{1},u_{2},r)\in H_{0}^{1}([0,1])\times H_{0}^{1}([0,1])\times H_{0}^{1}([0,1])$,
\begin{center}
$J_{*}(u_{1},u_{2},r)=|\nabla r|_{2}^{2}-\mathcal{R}\int_{0}^{1}\int_{0}^{1}F^{'}(\tau u_{1}+(1-\tau)u_{2})rd\tau\bar{r}dx$
\end{center}
and
\begin{center}
$J(u_{1},u_{2},r)=|\nabla r|_{2}^{2}-\mathcal{R}\int_{0}^{1}\int_{0}^{1}F^{'}(\tau u_{1}+(1-\tau)u_{2})rd\tau\bar{r} dx+G_{1}\left(\sum_{i=1}^{2}H^{\sigma}(u_{i})\right)|r|_{2}^{2}$,
\end{center}
where $F(u(t))=|u(t)|^{2\sigma}u(t)$ and $G_{1}$ is a constant to be determined. By Sobolev's embedding inequality, there exists $C$ such that
\begin{center}
$\mathcal{R}\int_{0}^{1}\int_{0}^{1}F^{'}(\tau u_{1}+(1-\tau)u_{2})rd\tau\bar{r}dx\leq C\left(\|u_{1}\|_{1}^{2\sigma}+\|u_{2}\|_{1}^{2\sigma}\right)|r|_{2}^{2}$.
\end{center}
Therefore, by (\ref{s4}), we can choose $G_{1}>0$ such that
\begin{center}
$J(u_{1},u_{2},r)=|\nabla r|_{2}^{2}-\mathcal{R}\int_{0}^{1}\int_{0}^{1}F^{'}(\tau u_{1}+(1-\tau)u_{2})rd\tau\bar{r} dx+G_{1}\left(\sum_{i=1}^{2}H^{\sigma}(u_{i})\right)|r|_{2}^{2}\geq \frac{1}{2}|\nabla r|_{2}^{2}$.
\end{center}
When $\lambda=-1$, we define for any $(u_{1},u_{2},r)\in H_{0}^{1}([0,1])\times H_{0}^{1}([0,1])\times H_{0}^{1}([0,1])$,
\begin{center}
$J(u_{1},u_{2},r)=|\nabla r|_{2}^{2}+\mathcal{R}\int_{0}^{1}\int_{0}^{1}F^{'}(\tau u_{1}+(1-\tau)u_{2})rd\tau\bar{r}dx$.
\end{center}
Note that
\begin{center}
$J(u_{1},u_{2},r)\geq \frac{1}{2}|\nabla r|_{2}^{2}$.
\end{center}
For any $N\geq 1$, we define for any $(u_{1},u_{2},r)\in H_{0}^{1}([0,1])\times H_{0}^{1}([0,1])\times H_{0}^{1}([0,1])$,
\begin{center}
$J_{FP}^{N}(u_{1},u_{2},r)=\exp\left(2\alpha t-\frac{\Lambda}{N^{\frac{1}{4}}}\int_{0}^{t}l\left(u_{1}(s),u_{2}(s)\right)ds\right)J(u_{1},u_{2},r)$,
\end{center}
where $l(u_{1}(s),u_{2}(s))=1+\sum_{i=1}^{2}H^{3\sigma+1}(u_{i})$ and $\Lambda$ is a constant.

In the following, we give the definition of the mild solution.
\begin{de}
The linear group $\{S(t)\}_{t\in \mathds{R}}$ is defined by $S(t)=e^{-itA}$, $t\in \mathds{R}$, associated to the equation $du=-iAudt$. We say $u$ is a $mild \ solution$ of (\ref{a2}), if
\begin{center}
$u(t)=S(t)u_{0}+i\lambda \int_{0}^{t}S(t-s)|u(s)|^{2\sigma}u(s)ds-\alpha\int_{0}^{t}S(t-s)u(s)ds+\int_{0}^{t}S(t-s)bdW(s)$  \quad  $\mathbb{P}$-a.s.
\end{center}
for all $t\geq 0$.
\end{de}
The well-posedness of equation (\ref{a2}) can be easily proved. Indeed, because the nonlinear part is not Lipschitz, we have to use a truncation argument. And by the fixed point theorem, we can prove the existence and uniqueness of the mild solution. Its proof is the same as the proof in \cite{BD2003}.

We denote by $\{P_{t}\}_{t\in \mathds{R}^{+}}$ the Markov semi-group associated to the solution of (\ref{a2}) and $\{P_{t}^{*}\}_{t\in \mathds{R}^{+}}$ the conjugate operator of $\{P_{t}\}_{t\in \mathds{R}^{+}}$.

\subsection{Basic properties of couplings}

We now recall some basic results of couplings. See, e.g. \cite{DO, M, O}.

Let $E$ be a Polish space, i.e. a complete separable metric space. Let $\mu_{1}$, $\mu_{2}$ be two distributions on a space $(E,\mathcal{E})$, where $\mathcal{E}$ is a $\sigma$-algebra of subsets of $E$. And let $Z_{1}$, $Z_{2}$ be two random variables $(\Omega,\mathcal{F})\rightarrow(E,\mathcal{E})$. We say that $(Z_{1},Z_{2})$ is a coupling of $(\mu_{1},\mu_{2})$ if $\mu_{i}=\mathcal{D}(Z_{i})$ for $i=1,2$, where we denote by $\mathcal{D}(Z_{i})$ the law of the random variable $Z_{i}$.

We denote by $Lip_{b}(E)$ the space of bounded and Lipschitz real valued functions on $E$ endowed with norm
\begin{center}
$\|\varphi\|_{L}=|\varphi|_{\infty}+L_{\varphi}$  \quad for any $\varphi\in Lip_{b}(E)$,
\end{center}
where $|\cdot|_{\infty}$ is the $L^{\infty}$ norm and $L_{\varphi}$ is the Lipschitz constant of $\varphi$. Let $\mathcal{P}(E)$ be the space of probability measures on $E$ endowed with the total variation metric
\begin{center}
$\|\mu\|_{var}=\sup\{\left|\mu(\Gamma)\right| | \ \Gamma\in \mathcal{B}(E)\}$ \quad for any $\mu\in\mathcal{P}(E)$,
\end{center}
where $\mathcal{B}(E)$ is the set of the Borelian subsets of $E$. And $\|\cdot\|_{var}$ is the dual norm of $|\cdot|_{\infty}$. We also use a Wasserstein norm
\begin{center}
$\|\mu\|_{W}=\sup\limits_{\varphi\in Lip_{b}(E),\|\varphi\|_{L}\leq 1}\left|\int_{E}\varphi(u)d\mu(u)\right|$ \quad for any $\mu\in\mathcal{P}(E)$
\end{center}
which is the dual norm of $\|\cdot\|_{L}$.

Let $\mu$, $\mu_{1}$, $\mu_{2}\in\mathcal{P}(E)$, and let $\mu_{1}$ and $\mu_{2}$ be absolutely continuous with respect to $\mu$. We set
\begin{align*}
d(\mu_{1}\wedge\mu_{2})=\left(\frac{d\mu_{1}}{d\mu}\wedge\frac{d\mu_{2}}{d\mu} \right )d\mu.
\end{align*}
This definition does not depend on the choice of $\mu$. And we have
\begin{align*}
\|\mu_{1}-\mu_{2}\|_{var}=\frac{1}{2}\int_{E}\left|\frac{d\mu_{1}}{d\mu}-\frac{d\mu_{2}}{d\mu}\right|d\mu.
\end{align*}
Note that if $\mu_{1}$ is absolutely continuous with respect to $\mu_{2}$, then we have
\begin{align}\label{a10}
\|\mu_{1}-\mu_{2}\|_{var}\leq \frac{1}{2} \sqrt{\int\left(\frac{d\mu_{1}}{d\mu_{2}}\right)^{2}d\mu_{2}-1}.
\end{align}

\begin{lem}[\cite{DO, M, O}]\label{lemma2.4}
Let $\mu_{1}$, $\mu_{2}$ be probability measures on $(E,\mathcal{E})$. Then $\|\mu_{1}-\mu_{2}\|_{var}=\min \mathbb{P}(Z_{1}\neq Z_{2})$,
the minimum is taken over all couplings $(Z_{1},Z_{2})$ of $(\mu_{1},\mu_{2})$. A coupling $(Z_{1},Z_{2})$ is said to be maximal if $\|\mu_{1}-\mu_{2}\|_{var}=\mathbb{P}(Z_{1}\neq Z_{2})$ and it has the property:
\begin{center}
$\mathbb{P}\left(Z_{1}=Z_{2},Z_{1}\in\Gamma\right)=(\mu_{1}\wedge\mu_{2})(\Gamma)$ \quad for any $\Gamma\in\mathcal{E}$.
\end{center}
\end{lem}
\begin{prop}[\cite{DO, O}]\label{proposition2.5}
Let $E$ and $F$ be Polish spaces, $\mu_{1}$, $\mu_{2}$ a pair of probability measures on
$E$ and $f_{0}:E\rightarrow F$ a measurable mapping. We set
$\nu_{i}=f_{0}^{*}\mu_{i}$, \ $i=1,2$.
Then there exists a coupling $(V_{1},V_{2})$ for $(\mu_{1},\mu_{2})$ such that $(f_{0}(V_{1}),f_{0}(V_{2}))$ is a maximal coupling for $(\nu_{1},\nu_{2})$.
\end{prop}

\subsection{Main results}

In this section, we will state the main results.

We first define $G$ by $Gu=iAu$ and set
\begin{align*}
&X=P_{N_{*}}u, \ Y=Q_{N_{*}}u, \ \beta=P_{N_{*}}W, \ \eta=Q_{N_{*}}W, \\
&\sigma_{l}=P_{N_{*}}bP_{N_{*}}, \ \sigma_{h}=Q_{N_{*}}bQ_{N_{*}}, \\
&f(X,Y)=-i\lambda P_{N_{*}}\left(|X+Y|^{2\sigma}(X+Y)\right), \ g(X,Y)=-i\lambda Q_{N_{*}}\left(|X+Y|^{2\sigma}(X+Y)\right).
\end{align*}
Then (\ref{a2}) can be written in the form
\begin{equation}\label{a3}
\begin{cases}
dX+GXdt+\alpha Xdt+f(X,Y)dt=\sigma_{l}d\beta, \\
dY+GYdt+\alpha Ydt+g(X,Y)dt=\sigma_{h}d\eta, \\
X(0)=x_{0}, Y(0)=y_{0}.
\end{cases}
\end{equation}
We note that Assumption \ref{s2} implies that $\sigma_{l}$ is invertible.

Given two initial datas $u_{0}^{i}=\left(x_{0}^{i},y_{0}^{i}\right)$, $i=1,2$, we construct a coupling
\begin{center}
$(u_{1},u_{2})=\left((X_{1},Y_{1}),(X_{2},Y_{2})\right)$
\end{center}
of the two solutions $u(\cdot,u_{0}^{i})=\left(X(\cdot,x_{0}^{i}),Y(\cdot,y_{0}^{i})\right)$, $i=1,2$, of (\ref{a3}).
We denote by $l_{0}$ an integer valued random process, which is particularly convenient when deriving properties of the coupling:
\begin{center}
$l_{0}(k)=\min\left \{l\in\{0,\ldots,k \right\}|\ (P_{l,k}) \ {\rm holds}\}$,
\end{center}
where $\min\emptyset=\infty$ and
\begin{equation*}
(P_{l,k}):
\begin{cases}
X_{1}(t)=X_{2}(t), \ \eta_{1}(t)=\eta_{2}(t), \qquad \forall t\in[lT,kT] \\
H_{l}\leq d_{0}, \\
E_{u_{i},3\sigma+1}(t,lT)\leq \kappa+1+d_{0}^{3\sigma+1}+d_{0}^{6\sigma+2}+B(t-lT), \quad i=1,2,  \qquad \forall t\in[lT,kT],
\end{cases}
\end{equation*}
where $d_{0}$, $\kappa$, $B$ are constants and we set
\begin{center}
$H_{l}=H(u_{1}(lT))+H(u_{2}(lT))$.
\end{center}
We say that $X_{1}$, $X_{2}$ are coupled at $kT$ if $l_{0}(k)\leq k$, i.e. if $l_{0}(k)\neq \infty$.
The following properties hold for the integer valued random process $l_{0}$.
\begin{itemize}
\item[{\textbf{(H1)}}]
$l_{0}(k+1)=l$ implies $l_{0}(k)=l$ for any $l\leq k$, \nonumber \\
$l_{0}(k)\in\{0,1,\ldots,k\}\cup\{\infty\}$, \\ \nonumber
$l_{0}(k)$ depends only on $u_{1}|_{[0,kT]}$ and $u_{2}|_{[0,kT]}$, \\ \nonumber
$l_{0}(k)=k$ implies $H_{k}\leq d_{0}$.
\end{itemize}
\quad We now give four conditions on the coupling which allow to prove polynomial convergence to equilibrium.
\begin{itemize}
\item[{\textbf{(H2)}}]
There exist $C_{0}$ and $q>0$ such that for any $t\in[lT,kT]\cap \mathds{R}^{+}$, we have  \\
\indent $\mathbb{P}\left(d_{E}\left(u_{1}(t),u_{2}(t)\right)>C_{0}(t-lT)^{-q} \ {\rm and} \ l_{0}(k)\leq l\right)\leq C_{0}(t-lT)^{-q}$.
\end{itemize}
(H2) implies that if $u_{1}(t)$ and $u_{2}(t)$ are coupled at time $lT$, then the probability that the distance between $u_{1}(t)$ and $u_{2}(t)$ is small when $t>lT$ will be large.
\begin{itemize}
\item[\textbf{(H3)}]
For any $q \in \mathds{N} \backslash \{0\}$, there exists $T_{q}>0$ such that for any $l\leq k$, $T\geq T_{q}$, we have  \\
\indent $\mathbb{P}(l_{0}(k+1)\neq l \ | \ l_{0}(k)=l)\leq \frac{1}{2}\left(1+(k-l)T\right)^{-q}$.
\end{itemize}
We can deduce from (H3) that if $u_{1}(t)$ and $u_{2}(t)$ are coupled on $[lT,kT]$, then the probability that $u_{1}(t)$ and $u_{2}(t)$ decouple will be small. Moreover, if the time they have been coupled is longer, then the probability will be smaller.
\begin{itemize}
\item[\textbf{(H4)}]
For any $R_{0}$, $d_{0}>0$, there exist $T^{*}(R_{0},d_{0})>0$ and $p_{-1}(d_{0})>0$ such that for any $T\geq T^{*}(R_{0},d_{0})$, we have \\
\indent $\mathbb{P}(l_{0}(k+1)=k+1 \ | \ l_{0}(k)=\infty, H_{k}\leq R_{0})\geq p_{-1}(d_{0})$.
\end{itemize}
(H4) says that in a small ball, the probability that $u_{1}(t)$ and $u_{2}(t)$ will be coupled is positive.
\begin{itemize}
\item[\textbf{(H5)}]
There exists $C_{k}^{'}$ such that for any initial data $u_{0}$ and any stopping time $\tau\in\{kT, k\in \mathds{N}\}\cup \{\infty\}$, we have the estimates\\
\indent \quad $\mathbb{E}H(u(t))\leq \exp(-\alpha t)H(u_{0})+\frac{C_{k}^{'}}{2}$,  \\
\indent \quad $\mathbb{E}(H(u(\tau))|_{\tau < \infty})\leq C_{k}^{'}(H(u_{0})+\mathbb{E}(\tau|\tau<\infty))$.
\end{itemize}
In our case, $H(u)$ is the Lyapunov function. (H5) describes the Lyapunov structure.

We say the process $V=\left(u_{1},u_{2}\right)$ is $l_{0}$-Markov if the laws of $V(kT+\cdot)$ and of $l_{0}(k+\cdot)-k$ on $\{l_{0}(k)\in \{k,\infty\}\}$ conditioned by $\mathcal{F}_{kT}$ only depend on $V(kT)$ and equal to the laws of $V(\cdot,V(kT))$ and $l_{0}$, respectively.

In the proof, we construct a coupling $((u_{1},W_{1}),(u_{2},W_{2}))$ of two solutions which is $l_{0}$-Markov. We can modify the construction such that it is Markov at discrete times $T\mathds{N}=\{kT,k\in \mathds{N}\}$. But it seems impossible to modify the coupling to be Markov at continuous time.
\begin{thm}\label{a8}
There exists $N_{0}$ such that if Assumptions \ref{s2}-\ref{s3} hold with $N_{*}\geq N_{0}$, then for any $\left(u_{0}^{1},W_{0}^{1}\right), \left(u_{0}^{2},W_{0}^{2}\right)$, there exists a coupling $V=\left((u_{1},W_{1}),(u_{2},W_{2})\right)$ of the laws of $\left(u\left(\cdot, u_{0}^{1}\right),W\left(\cdot, W_{0}^{1}\right)\right)$ and $\left(u\left(\cdot, u_{0}^{2}\right),W\left(\cdot, W_{0}^{2}\right)\right)$, where $V$ is $l_{0}$- Markov and satisfies (H1)-(H5) with $R_{0}>4C_{1}^{'}$ and $R_{0}\geq d_{0}$. Furthermore, there exists $C>0$, such that for any $\varphi\in Lip_{b}(H_{0}^{1}([0,1]))$ and $u_{0}^{1},u_{0}^{2}\in H_{0}^{1}([0,1])$,
\[
\left|\mathbb{E}\varphi\left(u(t,u_{0}^{1})\right)-\mathbb{E}\varphi\left(u\left(t,u_{0}^{2}\right)\right)\right|\leq C\left(1+t\right)^{-q}\|\varphi\|_{L}\left(1+H(u_{0}^{1})+H(u_{0}^{2})\right). \]
\end{thm}
Based on Theorem \ref{a8}, we can easily obtain the following corollary.
\begin{cor}\label{corollary3.11}
Under the assumptions of Theorem \ref{a8}, there exist a constant $K_{1}>0$ and a unique invariant measure $\upsilon$ of $(P_{t})_{t>0}$ on $H_{0}^{1}([0,1])$. It satisfies
\[
\int_{H_{0}^{1}([0,1])}H(u)d\upsilon(u)\leq\frac{K_{1}}{2}.
\]
Furthermore, for any $\mu\in \mathcal{P}(H_{0}^{1}([0,1]))$, there exists $C>0$ such that
\[
\|P_{t}^{*}\mu-\nu\|_{W}\leq C(1+t)^{-q}\left(1+\int_{H_{0}^{1}([0,1])}H(u)d\mu(u)\right).
\]
\end{cor}

\section{Priori estimates and proof of Theorem \ref{a8}}
\setcounter{equation}{0}
\subsection{Priori estimates}
In this section, we will give some prior estimates needed to prove the main theorem.
\begin{prop}\label{x1}
There exists a measurable map
\begin{center}
$\Phi:C\left([0,T];P_{N_{*}}H^{1}_{0}\right)\times C\left([0,T];Q_{N_{*}}H^{-1}\right)\times H^{1}_{0}\rightarrow C\left([0,T];Q_{N_{*}}H^{1}_{0}\right)$
\end{center}
such that for any $(u,W)$ which is the weak solution of (\ref{a3}),
\begin{center}
$Y=\Phi\left(X,\eta,u_{0}\right)$ \ on  \ $[0,T]$.
\end{center}
Moreover, $\Phi$ is a non-anticipative function of $\left(X,\eta\right)$.
\end{prop}

We can rewrite the second equation of (\ref{a3}) in the following form
\[
Y(t)=S(t)y_{0}-\alpha\int_{0}^{t}S(t-s)Y(s)ds-\int_{0}^{t}S(t-s)g(X(s),Y(s))ds+\int_{0}^{t}S(t-s)\sigma_{h}d\eta.
\]
Given $u_{0}\in H_{0}^{1}$, $X\in C([0,T];P_{N_{*}}H_{0}^{1})$ and $\eta\in C([0,T];Q_{N_{*}}H^{-1})$, Proposition \ref{x1} can be proved by applying the fixed point theorem.

\begin{lem}\label{lemma3.1}
For each $k\in \mathds N \setminus\{0\}$ and $t\in \mathds R ^{+}$,
there exist $C$, $C_{1}$, $C_{k}>0$ such that
\begin{flalign*}
d\left|u\right|_{2}^{2+ \frac{4\sigma}{2-\sigma}}+\alpha\left(\frac{3}{2}+\frac{4\sigma}{2-\sigma}\right)|u|_{2}^{2+ \frac{4\sigma}{2-\sigma}}dt\leq \left(2+\frac{4\sigma}{2-\sigma}\right)|u|_{2}^{\frac{4\sigma}{2-\sigma}}(u,bdW)+Cdt.
\end{flalign*}
When $\lambda=1$, we have the estimates
\begin{flalign*}
&dH(u)+\alpha H(u)dt\leq \left(Au -|u|^{2\sigma}u,bdW\right)+G\left(2+\frac{4\sigma}{2-\sigma}\right)|u|_{2}^{\frac{4\sigma}{2-\sigma}}(u,bdW)+C_{1}dt, \\
&dH^{k}(u)+\frac{1}{2}\alpha kH^{k}(u)dt\leq kH^{k-1}(u)\left[\left(Au -|u|^{2\sigma}u,bdW\right)+G\left(2+\frac{4\sigma}{2-\sigma}\right)|u|_{2}^{\frac{4\sigma}{2-\sigma}}(u,bdW)\right]+C_{k}dt.
\end{flalign*}
When $\lambda=-1$, we similarly have the estimates
\begin{flalign*}
&dH(u)+\alpha H(u)dt\leq \left(Au +|u|^{2\sigma}u,bdW\right)+C_{1}dt, \\
&dH^{k}(u)+\frac{1}{2}\alpha kH^{k}(u)dt\leq kH^{k-1}(u)\left(Au+|u|^{2\sigma}u,bdW\right)+C_{k}dt.
\end{flalign*}
\end{lem}

\begin{proof}
\indent Using It\^o's formula to $|u|_{2}^{2+ \frac{4\sigma}{2-\sigma}}$, we obtain the estimate
\begin{eqnarray*}
d|u|_{2}^{2+ \frac{4\sigma}{2-\sigma}}=\left(2+\frac{4\sigma}{2-\sigma}\right)|u|_{2}^{\frac{4\sigma}{2-\sigma}}\left(u,-iAu
+i\lambda |u|^{2\sigma}u-\alpha u\right)dt+\left(2+\frac{4\sigma}{2-\sigma}\right)|u|_{2}^{\frac{4\sigma}{2-\sigma}}(u,bdW)+\Rmnum{1}+\Rmnum{2},
\end{eqnarray*}
where
\begin{align*}
& \Rmnum{1}=\frac{1}{2}\left(2+\frac{4\sigma}{2-\sigma}\right)\left(\frac{4\sigma}{2-\sigma}\right)|u|_{2}^{ \frac{4\sigma}{2-\sigma}-2}\left|b^{*}u\right|_{2}^{2}dt, \\
& \Rmnum{2}=\frac{1}{2}B_{0}\left(2+\frac{4\sigma}{2-\sigma}\right)\left|u\right|_{2}^{\frac{4\sigma}{2-\sigma}}dt.
\end{align*}
Young's inequality implies
\begin{align*}
\Rmnum{1}+\Rmnum{2}\leq C\left[\left(2+\frac{4\sigma}{2-\sigma}\right)\left(\frac{2\sigma}{2-\sigma}\right)+\left(1+\frac{2\sigma}{2-\sigma}\right)\right]B_{0}|u|_{2}^{\frac{4\sigma}{2-\sigma}}dt
\leq \frac{\alpha}{2}|u|_{2}^{2+ \frac{4\sigma}{2-\sigma}}dt+Cdt.
\end{align*}
It follows from the last inequality that
\begin{align}\label{g1}
d|u|_{2}^{2+ \frac{4\sigma}{2-\sigma}}+\alpha\left(\frac{3}{2}+\frac{4\sigma}{2-\sigma}\right)|u|_{2}^{2+ \frac{4\sigma}{2-\sigma}}dt\leq \left(2+\frac{4\sigma}{2-\sigma}\right)|u|_{2}^{\frac{4\sigma}{2-\sigma}}(u,bdW)+Cdt.
\end{align}
\\
\textbf{The first case:} $\mathbf{\lambda=1}$.

Applying It\^o's formula to $H^{*}(u)$, we find
\begin{align*}
dH^{*}(u)
=&\left(Au-|u|^{2\sigma}u,-iAu+i|u|^{2\sigma}u-\alpha u\right)dt +\left(Au-|u|^{2\sigma}u,bdW\right) \\
&-\frac{1}{2}\sum\limits_{n=1}^{\infty}b_{n}^{2}\int_{0}^{1}\left(2\sigma|u|^{2\sigma-2}\left(\mathcal{R}(u\bar{e}_{n})\right)^{2}+|u|^{2\sigma}|e_{n}|^{2}\right)dxdt+\frac{1}{2}B_{1}dt\\
\leq& \left(-\alpha \|u\|_{1}^{2}+\alpha |u|_{2\sigma +2}^{2\sigma +2}\right)dt +\left(Au -|u|^{2\sigma}u,bdW\right)+\frac{1}{2}B_{1}dt.
\end{align*}
Consequently
\begin{align}\label{g2}
dH^{*}(u)+\left(\alpha \|u\|_{1}^{2}-\alpha |u|_{2\sigma +2}^{2\sigma +2}\right)dt \leq \left(Au -|u|^{2\sigma}u,bdW\right)+\frac{1}{2}B_{1}dt.
\end{align}
Employing (\ref{g1})-(\ref{g2}), we deduce
\begin{align*}
&dH(u)+\left(\alpha \|u\|_{1}^{2}-\alpha |u|_{2\sigma +2}^{2\sigma +2}+G\alpha\left(\frac{3}{2}+\frac{4\sigma}{2-\sigma}\right)|u|_{2}^{2+ \frac{4\sigma}{2-\sigma}}\right)dt \\
&\leq \left(Au -|u|^{2\sigma}u,bdW\right)+G\left(2+\frac{4\sigma}{2-\sigma}\right)|u|_{2}^{\frac{4\sigma}{2-\sigma}}(u,bdW)
+C_{1}dt.
\end{align*}
Using Gagliardo-Nirenberg's inequality, we have
\begin{align*}
&\alpha \|u\|_{1}^{2}-\alpha |u|_{2\sigma +2}^{2\sigma +2}+G\alpha\left(\frac{3}{2}+\frac{4\sigma}{2-\sigma}\right)|u|_{2}^{2+ \frac{4\sigma}{2-\sigma}} \\
&\geq \alpha \|u\|_{1}^{2}- \frac{\alpha}{2\sigma +2}\|u\|_{1}^{2}-\frac{G\alpha}{2}|u|_{2}^{2+ \frac{4\sigma}{2-\sigma}}+G\alpha\left(\frac{3 }{2}+\frac{4\sigma}{2-\sigma}\right)|u|_{2}^{2+ \frac{4\sigma}{2-\sigma}} \\
&\geq \alpha\frac{2\sigma +1}{2\sigma +2}\|u\|_{1}^{2}+G\alpha\left(1+\frac{4\sigma}{2-\sigma}\right)|u|_{2}^{2+ \frac{4\sigma}{2-\sigma}} \\
&\geq \alpha H(u).
\end{align*}
So
\begin{align}
dH(u)+\alpha H(u)dt\leq \left(Au -|u|^{2\sigma}u,bdW\right)+G\left(2+\frac{4\sigma}{2-\sigma}\right)|u|_{2}^{\frac{4\sigma}{2-\sigma}}(u,bdW)+C_{1}dt.
\end{align}
We apply It\^{o}'s formula to $H^{k}(u)$, then
\begin{align*}
dH^{k}(u)\leq &kH^{k-1}(u)\left[-\alpha H(u)dt+\left(Au-|u|^{2\sigma}u,bdW\right)+G\left(2+\frac{4\sigma}{2-\sigma}\right)|u|_{2}^{\frac{4\sigma}{2-\sigma}}(u,bdW)+C_{1}dt\right] \\
&+\frac{1}{2}k(k-1)H^{k-2}(u)d\langle M_{1}\rangle,
\end{align*}
where
\[
dM_{1}=\left(Au-|u|^{2\sigma}u,bdW\right)+G\left(2+\frac{4\sigma}{2-\sigma}\right)|u|_{2}^{\frac{4\sigma}{2-\sigma}}(u,bdW),
\]
\begin{align}\label{g5}
d\langle M_{1}\rangle\leq C\left(B_{1}\|u\|_{1}^{2}+B_{1}|u|_{2\sigma+2}^{2(2\sigma+1)}+B_{0}|u|_{2}^{\frac{8\sigma}{2-\sigma}+2}\right)dt
\leq Cdt+2\varepsilon_{1}\alpha\frac{1}{k} H^{2}(u)dt.
\end{align}
Applying Young's inequality, we compute
\begin{align*}
&C\left[kH^{k-1}\left(u\right)+\frac{1}{2}k(k-1)H^{k-2}(u)\right]dt+\alpha k\varepsilon_{1} H^{k}(u)dt \\
&\leq \varepsilon \alpha kH^{k}(u)dt+\alpha k\varepsilon_{1} H^{k}(u)dt+C_{k}dt.
\end{align*}
We choose $\varepsilon$, $\varepsilon_{1}$ so small that $\varepsilon+\varepsilon_{1}\leq\frac{1}{2}$.
Hence
\begin{align}\label{g6}
dH^{k}(u)+\frac{1}{2}\alpha kH^{k}(u)dt\leq kH^{k-1}(u)dM_{1}+C_{k}dt.
\end{align}
\\
\textbf{The second case:} $\mathbf{\lambda=-1}$.

Using It\^o's formula to $H(u)$, we find
\begin{align*}
dH(u)
=&\left(Au+|u|^{2\sigma}u,-iAu-i|u|^{2\sigma}u-\alpha u\right)dt +\left(Au+|u|^{2\sigma}u,bdW\right) \\
&+\frac{1}{2}\sum\limits_{n=1}^{\infty}b_{n}^{2}\int_{0}^{1}\left(2\sigma|u|^{2\sigma-2}(\mathcal{R}(u\bar{e}_{n}))^{2}+|u|^{2\sigma}|e_{n}|^{2}\right)dxdt+\frac{1}{2}B_{1}dt\\
\leq& \left(-\alpha \|u\|_{1}^{2}-\alpha |u|_{2\sigma +2}^{2\sigma +2}\right)dt +\left(Au +|u|^{2\sigma}u,bdW\right)+\frac{1}{2}B_{0}\left(2\sigma+1\right)|u|_{2\sigma}^{2\sigma}dt+\frac{1}{2}B_{1}dt.
\end{align*}
We infer from Young's inequality and Sobolev's embedding inequality that
\[
dH(u)\leq \left(-\alpha \|u\|_{1}^{2}-\alpha |u|_{2\sigma +2}^{2\sigma +2}\right)dt +\left(Au +|u|^{2\sigma}u,bdW\right)+\frac{1}{2}\alpha |u|_{2\sigma+2}^{2\sigma+2}dt+C_{1}dt.
\]
Thus
\begin{align}\label{k1}
dH(u)+\alpha H(u)dt\leq \left(Au +|u|^{2\sigma}u,bdW\right)+C_{1}dt.
\end{align}
Next, we apply It\^o's formula to $H^{k}(u)$, we then have
\[
dH^{k}(u)+\alpha kH^{k}(u)dt\leq kH^{k-1}(u)\left(Au+|u|^{2\sigma}u,bdW\right)+kC_{1}H^{k-1}(u)dt+\frac{1}{2}k(k-1)H^{k-2}(u)d\langle M_{1}^{'}\rangle,
\]
where
\[
dM_{1}^{'}=\left(Au+|u|^{2\sigma}u,bdW\right),
\]
\begin{align}\label{k2}
d\langle M_{1}^{'}\rangle \leq B_{1}\|u\|_{1}^{2}dt+B_{1}|u|_{2\sigma+2}^{2(2\sigma+1)}dt
\leq 2\varepsilon_{2} \alpha\frac{1}{k} H^{2}(u)dt+Cdt.
\end{align}
Applying Young's inequality, we deduce
\begin{align*}
&C\left[kH^{k-1}(u)+\frac{1}{2}k(k-1)H^{k-2}(u)\right]dt+\alpha k\varepsilon_{2} H^{k}(u)dt \\
&\leq \alpha k\varepsilon_{3} H^{k}(u)dt+\alpha k\varepsilon_{2} H^{k}(u)dt+C_{k}dt.
\end{align*}
We choose $\varepsilon_{2}$, $\varepsilon_{3}$ which are very small, then
\begin{align}\label{k3}
dH^{k}(u)+\frac{1}{2}\alpha kH^{k}(u)dt\leq kH^{k-1}(u)\left(Au+|u|^{2\sigma}u,bdW\right)+C_{k}dt.
\end{align}
\end{proof}

\begin{lem}\label{lemma3.2}
For any $k\in \mathds N \setminus\{0\}$,
$t\in \mathds R ^{+}$ and stopping time $\tau$, there exists $C_{k}^{'}>0$ such that
\begin{align*}
&\mathbb{E}\left(H^{k}(u(t))\right)\leq \exp\left(-\frac{\alpha}{2}kt\right)H^{k}(u_{0})+ \frac{C_{k}^{'}}{2}, \\
&\mathbb{E}\left(H^{k}(u(\tau))\right)\leq H^{k}(u_{0})+C_{k}^{'}\mathbb{E}(\tau).
\end{align*}

\end{lem}

\begin{proof}
\textbf{The first case:} $\mathbf{\lambda=1}$.

Multiplying (\ref{g6}) by $\exp\left(\frac{1}{2}\alpha kt\right)$, we deduce
\[
d\left(\exp\left(\frac{1}{2}\alpha kt\right)H^{k}\left(u(t)\right)\right)\leq \exp\left(\frac{1}{2}\alpha kt\right)kH^{k-1}(u(t))dM_{1}(t)+C_{k}\exp\left(\frac{1}{2}\alpha kt\right)dt.
\]
We integrate the last inequality from $0$ to $t$ to find
\[
\exp\left(\frac{1}{2}\alpha kt\right)H^{k}(u(t))\leq H^{k}(u_{0})+\int_{0}^{t}\exp\left(\frac{1}{2}\alpha ks\right)kH^{k-1}(u(s))dM_{1}(s)+C_{k}\int_{0}^{t}\exp\left(\frac{1}{2}\alpha ks\right)ds.
\]
Hence
\[
H^{k}(u(t))\leq \exp\left(-\frac{1}{2}\alpha kt\right)H^{k}(u_{0})+\int_{0}^{t}\exp\left(-\frac{1}{2}\alpha k(t-s)\right)kH^{k-1}(u(s))dM_{1}(s)+C_{k}\frac{2}{\alpha k}.
\]
Taking the expectation, we have
\[
\mathbb{E}(H^{k}(u(t)))\leq \exp\left(-\frac{1}{2}\alpha kt\right)H^{k}(u_{0})+\frac{C_{k}^{'}}{2},
\]
which implies the first inequality of Lemma \ref{lemma3.2} holds.

We now assume that $M>0$ is a constant and $\tau<M$ is a bounded stopping time. Then integrating (\ref{g6}) from 0 to $\tau$ and taking the expectation, we compute
\[
\mathbb{E}\left(H^{k}(u(\tau))\right)\leq H^{k}(u_{0})+C_{k}^{'}\mathbb{E}(\tau).
\]
Therefore, the second inequality of Lemma \ref{lemma3.2} for bounded stopping times follows.

Assume that $\tau$ is a general stopping time. We consider the second inequality of Lemma \ref{lemma3.2} for the stopping time $\tau\wedge M$, we have
\[
\mathbb{E}\left(H^{k}(u(\tau\wedge M))\right)\leq H^{k}(u_{0})+C_{k}^{'}\mathbb{E}(\tau\wedge M).
\]
By Fatou's Lemma and lower semicontinuity, when $M\longrightarrow\infty$, we calculate
\begin{align*}
\mathbb{E}\left(H^{k}(u(\tau))\right)\leq
\liminf\limits_{{M\rightarrow\infty}}\mathbb{E}\left(H^{k}(u(\tau\wedge M))\right)
\leq \limsup\limits_{{M\rightarrow\infty}}\left(H^{k}(u_{0})+C_{k}^{'}\mathbb{E}(\tau\wedge M)\right)
\leq H^{k}(u_{0})+C_{k}^{'}\mathbb{E}(\tau),
\end{align*}
which yields the second inequality of Lemma \ref{lemma3.2}.

The similar argument holds for the second case: $\lambda=-1$.
\end{proof}

\begin{lem}\label{lemma3.3}
Suppose that $u$ is a solution of (\ref{a2}) associated with a Wiener process W. Then
for any $(k,p)\in(\mathds N \setminus\{0\})^{2}$, $\rho>0$ and $0\leq T<\infty$, we have the estimates
\begin{align*}
\mathbb{P}\left(\sup \limits_{t\in [0,T]}\left(E_{u,k}(t)-C_{k}^{'}t\right)\geq H^{k}(u_{0})+\rho\left(H^{2k}(u_{0})+ T\right)\right)
\leq K_{k,p}\rho^{-p}, \\
\mathbb{P}\left(\sup \limits_{t\in [T,\infty)}\left(E_{u,k}(t)-C_{k}^{'}t\right)\geq H^{k}(u_{0})+H^{2k}(u_{0})+ 1+\rho\right)
\leq K_{k,p}\left(\rho+T\right)^{-p},
\end{align*}
the constants $C_{k}^{'}$ and $K_{k,p}$ depending only on $k$ and $p$.
\end{lem}

\begin{proof}
We only prove the case $\lambda=1$; the other case $\lambda=-1$ is similar.

We first set \[
dM_{k}(t)=kH^{k-1}(u(t))dM_{1}(t).
\]
Taking into account (\ref{g5}), we see that
\[
d\langle M_{k}\rangle(t) \leq C_{k}\left(1+H^{2k}(u(t))\right)dt.
\]
Integrating (\ref{g6}) from $0$ to $t$ and taking the expectation, we find for any $k\geq 1$,
\[
\mathbb{E}\int_{0}^{t}H^{k}(u(s))ds\leq C_{k}\left(H^{k}(u_{0})+t\right).
\]
Therefore, for any $p\geq 1$,
\begin{align*}
\mathbb{E} \langle M_{k}\rangle^{p}(t)
&\leq\mathbb{E}\left(\int_{0}^{t}C_{k}\left(H^{2k}(u(s))+1\right)ds\right)^{p} \\
&\leq \left(\mathbb{E}\int_{0}^{t}C_{k}\left(H^{2k}(u(s))+1\right)ds\right)^{p} \\
&\leq 2^{p}C_{k}^{p}\left(t^{p}+\left(\mathbb{E}\int_{0}^{t}H^{2k}(u(s))ds\right)^{p}\right) \\
&\leq C_{k,p}\left(H^{2kp}(u_{0})+t^{p}\right),
\end{align*}
where the second inequality holds by the inverse H\"{o}lder inequality.

According to (\ref{g6}) and the martingale inequality, we get
\begin{align*}
\mathbb{P}\left(\sup \limits_{t\in [0,T]}\left(E_{u,k}(t)-C_{k}^{'}t\right)\geq H^{k}(u_{0})+\rho\left(H^{2k}(u_{0})+ T\right)\right)
&\leq \frac{\mathbb{E}\langle M_{k}\rangle^{\frac{p}{2}}}{\left(\rho(H^{2k}(u_{0})+ T)\right)^{p}} \\
\leq \frac{C_{k,\frac{p}{2}}\left(H^{kp}(u_{0})+T^{\frac{p}{2}}\right)}{\left(\rho(H^{2k}(u_{0})+ T)\right)^{p}}
&\leq K_{k,p}\rho^{-p}.
\end{align*}
This proves the first inequality of Lemma \ref{lemma3.3}.

We similarly have
\begin{align*}
\mathbb{P}\left(\sup \limits_{[n,n+1]}\left(E_{u,k}(t)-C_{k}^{'}t\right)\geq H^{k}(u_{0})+H^{2k}(u_{0})+ 1+\rho+n\right)
&\leq \frac{\mathbb{E}\langle M_{k}\rangle^{p+1}}{\left(H^{2k}(u_{0})+ 1+\rho+n\right)^{2p+2}} \\
\leq \frac{C_{k,p+1}\left(H^{2k(p+1)}(u_{0})+(n+1)^{p+1}\right)}{\left(H^{2k}(u_{0})+ 1+\rho+n\right)^{2p+2}}
&\leq \frac{C_{k,p+1}}{\left(H^{2k}(u_{0})+ 1+\rho+n\right)^{p+1}}.
\end{align*}
Then summing the last inequality over $n\geq T$, where $T$ is an integer, we see that for any $(k,p) \in (\mathds{N} \setminus\{0\})^{2}$, there exists $K_{k,p}$ such that
\begin{align*}
\mathbb{P}\left(\sup \limits_{t\in[T,\infty)}M_{k}(t)\geq H^{2k}(u_{0})+ 1+\rho +t\right)
\leq K_{k,p}\left(\rho+T\right)^{-p}.
\end{align*}
In view of (\ref{g6}), we deduce the second inequality of Lemma \ref{lemma3.3}.
\end{proof}

\begin{lem}\label{lemma3.4}
Assume $W_{1}$ and $W_{2}$
are two cylindrical Wiener process on $L^{2}([0,1])$. Suppose also $\left(u_{i},W_{i}\right)_{i=1,2}$ is a pair of solutions of (\ref{a2}). If $R_{0}\geq \left(\sum_{i=1}^{2}H(u_{0}^{i})\right)\vee C_{1}^{'}$, then
\[
\mathbb{P}\left(H\left(u_{1}(t)\right)+H(u_{2}(t))\geq 4C_{1}^{'}\right)\leq \frac{1}{2},
\]
provided $t\geq\theta_{1}\left(R_{0}\right)=\frac{2}{\alpha}\ln\frac{R_{0}}{C_{1}^{'}}$.
\end{lem}

\begin{proof}
It follows from Chebyshev's inequality and Lemma \ref{lemma3.2}.
\end{proof}

As in \cite{KS}, it is crucial to prove that the probability that a solution enters a ball of a small radius is controlled precisely. It is still true for the damped SNLS considered here. But its proof is more difficult than in the case of the Navier-Stokes equations.

\begin{lem}\label{proposition3.5}
Assume $R_{0}$, $R_{1}>0$. Then there exist $T_{-1}(R_{0},R_{1})\geq0$ and $\pi_{-1}(R_{1})>0$
such that
\[
\mathbb{P}\left(H\left(u(t,u_{0}^{1})\right)+H\left(u(t,u_{0}^{2})\right)\leq R_{1}\right)\geq \pi_{-1}\left(R_{1}\right),
\]
provided $H(u_{0}^{1})+H(u_{0}^{2})\leq R_{0}$ and $t\geq T_{-1}(R_{0},R_{1})$.
\end{lem}

\begin{proof}
According to Lemma \ref{lemma3.4}, it suffices to show Lemma \ref{proposition3.5} for $R_{0}=4C_{1}^{'}$ and $t=T_{-1}(R_{0},R_{1})$ (instead of $t\geq T_{-1}(R_{0},R_{1})$). Consequently, we only prove this lemma for $R_{0}=4C_{1}^{'}$.

Assume $T, \delta>0$. Applying Chebyshev's inequality, we deduce that there exists $N_{-2}=N_{-2}(T,\delta)\in \mathds{N}$ such that
\[
\mathbb{P}\left(\sup \limits_{t\in[0,T]}\left\|bQ_{N_{-2}}W(t)\right\|_{3}>\frac{\delta}{2}\right)\leq\frac{4}{\delta^{2}}\sum\limits_{n>N_{-2}}\mu_{n}^{3}b_{n}^{2}\leq \frac{1}{2}.
\]
Furthermore, since $P_{N_{-2}}W$ is a finite dimensional Brownian motion, we find
\[
\pi_{-3}\left(T,\delta,N_{-2}\right)=\mathbb{P}\left(\sup \limits_{t\in[0,T]}\left|P_{N_{-2}}W(t)\right|_{2}\leq \frac{\delta}{2}\left\|b\right\|_{\mathcal{L}_{2}\left(L^{2}([0,1]),H^{3}([0,1])\right)}^{^{-1}}\right)>0.
\]
Then we have
\[
\mathbb{P}\left(\sup \limits_{t\in[0,T]}\left\|bW(t)\right\|_{3}\leq\delta\right)\geq \mathbb{P}\left(\sup \limits_{t\in[0,T]}\left\|bQ_{N_{-2}}W(t)\right\|_{3}\leq\frac{\delta}{2}\right)\pi_{-3}\left(T,\delta,N_{-2}\right).
\]
Therefore
\[
\pi_{-2}(T,\delta)=\mathbb{P}\left(\sup \limits_{t\in[0,T]}\left\|bW(t)\right\|_{3}\leq \delta\right)>0.
\]
It suffices to prove that there exist $T_{-1}(R_{1}),\delta_{-1}(R_{1})>0$ such that
\begin{align}\label{g3}
\left\{\sup \limits_{t\in[0,T_{-1}]}\left\|bW(t)\right\|_{3}\leq\delta_{-1}\right\} \subset \left\{H\left(u(T_{-1},u_{0})\right)\leq \frac{1}{2}R_{1}\right\},
\end{align}
provided $H(u_{0})\leq \frac{1}{2}R_{0}$.

We turn now to prove (\ref{g3}). Let
\[
v=u(\cdot,u_{0})-bW.
\]
Then
\begin{align}\label{z2}
dv+\alpha vdt+iAvdt-i\lambda|v+bW|^{2\sigma}(v+bW)dt=\left(-\alpha-iA\right)bWdt.
\end{align}
Applying It\^o's formula to $|v|_{2}^{2}$, we find
\[
\frac{d|v|_{2}^{2}}{dt}+2\alpha |v|_{2}^{2}=\left(2v,i\lambda|v+bW|^{2\sigma}(v+bW)+(-\alpha-iA)bW\right).
\]
Since
$\left(v,i|v+bW|^{2\sigma}v\right)=0$, we deduce
\begin{align*}
&\left(2v,i\lambda|v+bW|^{2\sigma}(v+bW)+(-\alpha-iA)bW\right) \\
&\leq C\left(|v|,|v|^{2\sigma}|bW|\right)+C\left(|v|,|bW|^{2\sigma+1}\right)+|\left(2v,-\alpha bW\right)|+|\left(2v,-iA(bW)\right)| \\
&\leq C\|bW\|_{3}\left(1+\|v\|_{1}^{2\sigma+1}\right)\left(1+\|bW\|_{3}^{2\sigma}\right).
\end{align*}
Using It\^o's formula to $|v|_{2}^{2+\frac{4\sigma}{2-\sigma}}$, we have
\begin{align*}
\frac{d|v|_{2}^{2+\frac{4\sigma}{2-\sigma}}}{dt}&=\left(2+\frac{4\sigma}{2-\sigma}\right)|v|_{2}^{\frac{4\sigma}{2-\sigma}}\left(v,-\alpha v-iA v+i\lambda|v+bW|^{2\sigma}(v+bW)+(-\alpha-iA)bW\right)  \\
&\leq -\alpha\left(2+\frac{4\sigma}{2-\sigma}\right)|v|_{2}^{2+\frac{4\sigma}{2-\sigma}}+C\|bW\|_{3}\left(1+\|bW\|_{3}^{2\sigma}\right)\left(1+\|v\|_{1}^{1+2\sigma+\frac{4\sigma}{2-\sigma}}\right).
\end{align*}
So
\begin{align}\label{g4}
\frac{d|v|_{2}^{2+\frac{4\sigma}{2-\sigma}}}{dt}+\alpha\left(2+\frac{4\sigma}{2-\sigma}\right)|v|_{2}^{2+\frac{4\sigma}{2-\sigma}}\leq C\|bW\|_{3}\left(1+\|bW\|_{3}^{2\sigma}\right)\left(1+\|v\|_{1}^{1+2\sigma+\frac{4\sigma}{2-\sigma}}\right).
\end{align}
Then we apply It\^o's formula to $H^{*}(v)$ yields
\[
\frac{dH^{*}(v)}{dt}+\alpha \|v\|_{1}^{2}=-\left(Av-\lambda|v|^{2\sigma}v,(\alpha+iA)bW\right)+\alpha\left(\lambda|v+bW|^{2\sigma}(v+bW),v\right).
\]
We write
\begin{align*}
\Rmnum{1}_{1}=\alpha\left(\left(\lambda|v+bW|^{2\sigma}(v+bW),v\right)-\lambda|v|_{2\sigma+2}^{2\sigma+2}\right)
=\alpha\lambda\left(|v+bW|^{2\sigma}(v+bW)-|v|^{2\sigma}v,v\right).
\end{align*}
Then
\begin{align}\label{g7}
\frac{dH^{*}(v)}{dt}+\alpha \|v\|_{1}^{2}-\alpha\lambda|v|_{2\sigma+2}^{2\sigma+2}=\Rmnum{1}_{1}+\Rmnum{1}_{2},
\end{align}
where
\[
\Rmnum{1}_{2}=-\left(Av-\lambda|v|^{2\sigma}v,(\alpha+iA)bW\right).
\]
Recall that for any $z,h\in \mathds{C}$,
\[
\left||z+h|^{2\sigma}(z+h)-|z|^{2\sigma}z\right|\leq C|h|\left(|z|^{2\sigma}+|h|^{2\sigma}\right).
\]
We use the last inequality and H\"{o}lder's inequality to find
\begin{align}\label{z1}
\Rmnum{1}_{1}+\Rmnum{1}_{2}&\leq \left|\left(-Av,(\alpha+iA)bW\right)\right|+\left|\left(\lambda|v|^{2\sigma}v,(\alpha+iA)bW\right)\right|+C\left(\left(|bW|^{2\sigma}+|v|^{2\sigma}\right)|bW|,|v|\right) \\ \nonumber
&\leq C\|bW\|_{3}\left(1+\|v\|_{1}^{2\sigma+1}\right)\left(1+\|bW\|_{3}^{2\sigma}\right).
\end{align}
\\
\textbf{The first case:} $\mathbf{\lambda=1}$.

Combining now (\ref{g4})-(\ref{z1}), we deduce
\begin{align*}
&\frac{dH(v)}{dt}+\alpha \|v\|_{1}^{2}-\alpha|v|_{2\sigma+2}^{2\sigma+2}+G\alpha\left(2+\frac{4\sigma}{2-\sigma}\right)|v|_{2}^{2+\frac{4\sigma}{2-\sigma}} \\
&\leq C\|bW\|_{3}\left(1+\|v\|_{1}^{2\sigma+1}\right)\left(1+\| bW\|_{3}^{2\sigma}\right)
+C\|bW\|_{3}\left(1+\|bW\|_{3}^{2\sigma}\right)\left(1+\|v\|_{1}^{1+2\sigma+\frac{4\sigma}{2-\sigma}}\right).
\end{align*}
Applying Gagliardo-Nirenberg's inequality and taking into account (\ref{s1}), we calculate
\begin{align*}
&\alpha \|
v\|_{1}^{2}-\alpha|v|_{2\sigma+2}^{2\sigma+2}+G\alpha\left(2+\frac{4\sigma}{2-\sigma}\right)|v|_{2}^{2+\frac{4\sigma}{2-\sigma}} \\
&\geq \alpha \|v\|_{1}^{2}-\alpha\left(\frac{1}{2\sigma+2}\|
v\|_{1}^{2}+\frac{G}{2}|v|_{2}^{2+\frac{4\sigma}{2-\sigma}}\right)+G\alpha\left(2+\frac{4\sigma}{2-\sigma}\right)|v|_{2}^{2+\frac{4\sigma}{2-\sigma}} \\
&\geq\alpha H(v).
\end{align*}
Hence
\begin{align}\label{g8}
\frac{dH(v)}{dt}+\alpha H(v)\leq C\|bW\|_{3}\left(1+\|bW\|_{3}^{2\sigma}\right)\left(1+\|v\|_{1}^{1+2\sigma+\frac{4\sigma}{2-\sigma}}\right).
\end{align}
Let us assume that $T, \delta, R_{1}^{'}>0$. We also suppose
\[
\sup \limits_{t\in[0,T]} \|bW(t)\|_{3}\leq\delta
\]
and set
\[
\tau=\inf \left\{t\in[0,T] \ | \ H(v)> 3R_{0}\right\}.
\]
Integrating (\ref{g8}) from $0$ to $t$, we have
\begin{align}\label{g9}
H(v(t))\leq \frac{1}{2}\exp\left(-\alpha t\right)R_{0}+\frac{C}{\alpha}\delta\left(1+\delta^{2\sigma}\right)\left(1+R_{0}^{\sigma+1+\frac{2\sigma}{2-\sigma}}\right),
\end{align}
provided $t\leq\tau$.
Then we choose $\delta_{-2}(R_{1}^{'})$ such that
\[
\frac{C}{\alpha}\delta\left(1+\delta^{2\sigma}\right)\left(1+R_{0}^{\sigma+1+\frac{2\sigma}{2-\sigma}}\right)\leq R_{1}^{'}\wedge R_{0}.
\]
for any $\delta\leq\delta_{-2}(R_{1}^{'})$. Thus from (\ref{g9}), it follows that
$\tau= T$ and that $H\left(v(T)\right)\leq 2R_{1}^{'}$,
provided $T\geq \frac{1}{\alpha}\ln\left(\frac{R_{0}}{2R_{1}^{'}}\right)$.
We remark that
\[
H(u(T))\leq C\left(H(bW(T))+H(v(T))\right)\leq C\left(\delta^{2}\left(1+\delta^{\frac{4\sigma}{2-\sigma}}\right)+R_{1}^{'}\right).
\]
Then we choose $\delta$ and $R_{1}^{'}$ sufficiently small to derive (\ref{g3}).

The proof for the case $\lambda=-1$ is similar, so we omit it.
\end{proof}

\begin{lem}\label{proposition3.6}
Assume for any $k_{0}>0$ and $N\in \mathds N \setminus\{0\}$,
$W_{1}$, $W_{2}$ are two cylindrical Wiener processes, h is an adapted process with continuous paths in
 $P_{N}L^{2}\left([0,1]\right)$, $u_{1}$ is a solution in $C\left([0,T];H_{0}^{1}([0,1])\right)$ of
\begin{equation*}
\begin{cases}
du_{1}+\alpha u_{1}dt+iAu_{1}dt-i\lambda |u_{1}|^{2\sigma}u_{1}dt=bdW_{1}+hdt \\
u_{1}(0)=u_{0}^{1},
\end{cases}
\end{equation*}
$u_{2}$ is the solution of (\ref{a2}) for $u_{0}=u_{0}^{2}$ and $W=W_{2}$ and $\tau$
is a stopping time. Suppose also that
\begin{align}\label{z8}
P_{N}u_{1}=P_{N}u_{2}, Q_{N}W_{1}=Q_{N}W_{2} \qquad   on \indent [0,\tau]
\end{align}
and
\begin{align}\label{z9}
\|h(t)\|_{1}^{2}\leq k_{0}\left(l(u_{1}(t)+u_{2}(t))\right)^{\frac{2\sigma+1}{3\sigma+1}} \qquad   on \indent [0,\tau],
\end{align}
then there exists $\Lambda >0$ depending only on $k_{0}$ such that
\begin{align}\label{b1}
\mathbb{E}\left[J_{FP}^{N}(u_{1},u_{2},r)(t\wedge \tau)\right]\leq J\left(u_{0}^{1},u_{0}^{2},r_{0}\right)  \qquad t>0,
\end{align}
where $r_{0}=u_{0}^{1}-u_{0}^{2}$.
\end{lem}

\begin{proof}
In light of (\ref{z8}), the difference of the two solutions $r=u_{1}-u_{2}=Q_{N}u_{1}-Q_{N}u_{2}$
satisfies the equation
\begin{align}\label{z4}
dr=-iArdt+i\lambda Q_{N}\left(|u_{1}|^{2\sigma}u_{1}-|u_{2}|^{2\sigma}u_{2}\right)dt-\alpha rdt.
\end{align}
Then using It\^o's formula to $|r|_{2}^{2}$, we compute
\begin{align*}
d|r|_{2}^{2}+2\alpha |r|_{2}^{2}dt=\left(2r,i\lambda\left(|u_{1}|^{2\sigma}u_{1}-|u|_{2}^{2\sigma}u_{2}\right)\right)dt.
\end{align*}
Since
\begin{align*}
\left||u_{1}|^{2\sigma}u_{1}-|u|_{2}^{2\sigma}u_{2}\right|\leq C\left(\sum\limits_{i=1}^{2}|u_{i}|^{2\sigma}\right)|r|,
\end{align*}
we find
\begin{align}\label{z3}
d|r|_{2}^{2}+2\alpha |r|_{2}^{2}dt \leq C\mathcal{R}\int_{0}^{1}\left(\sum\limits_{i=1}^{2}|u_{i}|^{2\sigma}\right)\left|r\right|^{2}dxdt
\leq C\left(\|u_{1}\|_{1}^{2\sigma}+\|u_{2}\|_{1}^{2\sigma}\right)|r|_{2}^{2}dt
\leq C\left(\sum\limits_{i=1}^{2} H^{\sigma}(u_{i})\right)|r|_{2}^{2}dt.
\end{align}
\textbf{The first case:} $\mathbf{\lambda=1}$.

As demonstrated in Lemma \ref{lemma3.1}, for $i=1,2$, we have
\begin{align*}
dH(u_{i})+\alpha H(u_{i})dt\leq \left(M^{i},bdW_{i}\right)+C_{1}dt+1_{i=1}\left(M^{i},h\right)dt,
\end{align*}
\begin{align*}
dH^{\sigma}(u_{i})\leq& -\frac{1}{2}\alpha \sigma H^{\sigma}(u_{i})dt+\sigma H^{\sigma-1}(u_{i})\left(M^{i},bdW_{i}\right)
+C_{\sigma}dt+\sigma H^{\sigma-1}(u_{i})1_{i=1}\left(M^{i},h\right)dt,
\end{align*}
where
\[
M^{i}=Au_{i}-|u_{i}|^{2\sigma}u_{i}+G\left(2+\frac{4\sigma}{2-\sigma}\right)|u_{i}|_{2}^{\frac{4\sigma}{2-\sigma}}u_{i}.
\]
Employing Sobolev's embedding inequality and H\"{o}lder's inequality, we deduce
\begin{align*}
\|M^{1}\|_{-1}\leq C\left(1+H(u_{1})\right)^{\frac{3\sigma+2}{2\sigma+4}}, \\
\left(M^{1},h\right)\leq C\left(1+\sum \limits_{i=1}^{2} H(u_{i})\right)^{2\sigma+2}.
\end{align*}
We set
\[
Z_{1}=\left(\sum \limits_{i=1}^{2}H^{\sigma}(u_{i})\right)\left|r\right|_{2}^{2}.
\]
In view of (\ref{z3}), we see
\begin{align*}
dZ_{1}=&\left( \sum \limits_{i=1}^{2}H^{\sigma}(u_{i})\right)d|r|_{2}^{2}+\left(\sum \limits_{i=1}^{2}dH^{\sigma}(u_{i})\right)|r|_{2}^{2} \\
\leq& -2\alpha Z_{1}dt+\sum \limits_{i=1}^{2}\sigma H^{\sigma-1}(u_{i})\left(M^{i},bdW_{i}\right)|r|_{2}^{2}+C\left(\sum\limits_{i=1}^{2}H^{\sigma}(u_{i})\right)^{2}|r|_{2}^{2}dt \\
&+C|r|_{2}^{2}dt+C\left(\sum\limits_{i=1}^{2}\sigma H^{\sigma-1}(u_{i})\right)\left(1+\sum\limits_{i=1}^{2}H(u_{i})\right)^{2\sigma+2}|r|_{2}^{2}dt.
\end{align*}
Hence
\begin{align}\label{z5}
dZ_{1}+2\alpha Z_{1}dt
\leq \sum \limits_{i=1}^{2}\sigma H^{\sigma-1}(u_{i})\left(M^{i},bdW_{i}\right)|r|_{2}^{2}+C\left(1+\sum \limits_{i=1}^{2}H^{3\sigma+1}(u_{i})\right)|r|_{2}^{2}dt.
\end{align}
We first set
\[
F(u)=\left|u\right|^{2\sigma}u
\]
and note that its derivatives
\begin{align*}
F^{'}(u)(v)=2\sigma|u|^{2\sigma-2}\mathcal{R}(\overline{u}v)u+|u|^{2\sigma}v=\left(\sigma+1\right)|u|^{2\sigma}v+\sigma|u|^{2\sigma-2}u^{2}\overline{v},
\end{align*}
\begin{align*}
F^{''}(u)(v,w)=
&\sigma(\sigma+1)|u|^{2\sigma-2}\overline{u}vw+\sigma(\sigma+1)|u|^{2\sigma-2}u\overline{v}w \\
&+\sigma(\sigma+1)|u|^{2\sigma-2}uv\overline{w}+\sigma(\sigma-1)|u|^{2\sigma-4}u^{3}\overline{vw},
\end{align*}
\begin{align*}
F^{'''}(u)(v,w,z)
=&\sigma(\sigma-1)(\sigma+1)|u|^{2\sigma-4}\overline{u}^{2}vwz+\sigma(\sigma-1)(\sigma+1)|u|^{2\sigma-4}u^{2}\overline{vw}z \\
&+\sigma(\sigma-1)(\sigma+1)|u|^{2\sigma-4}u^{2}\overline{vz}w+\sigma(\sigma-1)(\sigma+1)|u|^{2\sigma-4}u^{2}v\overline{wz} \\
&+\sigma^{2}(\sigma+1)|u|^{2\sigma-2}\overline{v}wz +\sigma^{2}(\sigma+1)|u|^{2\sigma-2}v\overline{w}z+\sigma^{2}(\sigma+1)|u|^{2\sigma-2}vw\overline{z} \\
&+\sigma(\sigma-1)(\sigma-2)|u|^{2\sigma-6}u^{4}\overline{vwz}.
\end{align*}
We can rewrite (\ref{z4}) in the following form
\begin{align}\label{z6}
dr+iArdt+\alpha rdt=iQ_{N}\int_{0}^{1}F^{'}\left(\tau u_{1}+(1-\tau)u_{2}\right)rd\tau dt.
\end{align}
Applying It\^o's formula to $J_{*}(u_{1},u_{2},r)$ yields
\begin{align*}
d&J_{*}(u_{1},u_{2},r) \\
=& -2\alpha J_{*}(u_{1},u_{2},r) \\
&-\mathcal{R}\int_{0}^{1}\int_{0}^{1}F^{''}\left(\tau u_{1}+(1-\tau)u_{2}\right)\tau \left(r,-\alpha u_{1}dt-iAu_{1}dt+i|u_{1}|^{2\sigma}u_{1}dt+bdW_{1}+hdt\right)d\tau \bar{r}dx \\
&-\mathcal{R}\int_{0}^{1}\int_{0}^{1}F^{''}(\tau u_{1}+(1-\tau)u_{2})(1-\tau) \left(r,-\alpha u_{2}dt-iAu_{2}dt+i|u_{2}|^{2\sigma}u_{2}dt+bdW_{2}\right)d\tau \bar{r}dx \\
&-\frac{1}{2}\sum\limits_{p=1}^{\infty}b_{p}^{2} \mathcal{R}\int_{0}^{1}\int_{0}^{1}F^{'''}\left(\tau u_{1}+(1-\tau)u_{2}\right)\left(r,\tau e_{p},\tau e_{p}\right)d\tau \bar{r}dx dt\\
&-\frac{1}{2}\sum\limits_{p,q=1}^{\infty}b_{p}b_{q} \mathcal{R}\int_{0}^{1}\int_{0}^{1}F^{'''}\left(\tau u_{1}+(1-\tau)u_{2}\right)\left(r,\tau e_{p},(1-\tau)e_{q}\right)d\tau \bar{r}dxd\left\langle\left(W_{1},e_{p}\right),\left(W_{2},e_{q}\right)\right\rangle \\
&-\frac{1}{2}\sum\limits_{p,q=1}^{\infty}b_{p}b_{q} \mathcal{R}\int_{0}^{1}\int_{0}^{1}F^{'''}\left(\tau u_{1}+(1-\tau)u_{2}\right)\left(r,(1-\tau) e_{q},\tau e_{p}\right)d\tau \bar{r}dxd\left\langle\left(W_{2},e_{q}\right),\left(W_{1},e_{p}\right)\right\rangle \\
&-\frac{1}{2}\sum\limits_{q=1}^{\infty}b_{q}^{2} \mathcal{R}\int_{0}^{1}\int_{0}^{1}F^{'''}\left(\tau u_{1}+(1-\tau)u_{2}\right)\left(r,(1-\tau) e_{q},(1-\tau) e_{q}\right)d\tau \bar{r}dxdt \\
:=&-2\alpha J_{*}(u_{1},u_{2},r)dt+\Rmnum{1}+\Rmnum{2}+\Rmnum{3}+\Rmnum{4}+\Rmnum{5}+\Rmnum{6}.
\end{align*}
Using H\"{o}lder's inequality and Sobolev's embedding inequality, we deduce
\begin{align*}
\Rmnum{1}=&-\mathcal{R}\int_{0}^{1}\int_{0}^{1}F^{''}\left(\tau u_{1}+(1-\tau)u_{2}\right)\tau \left(r,-\alpha u_{1}dt-iAu_{1}dt+i|u_{1}|^{2\sigma}u_{1}dt\right)d\tau \bar{r}dx \\
&-\mathcal{R}\int_{0}^{1}\int_{0}^{1}F^{''}\left(\tau u_{1}+(1-\tau)u_{2}\right)\left(r,\tau bdW_{1}\right)d\tau \bar{r}dx \\
&-\mathcal{R}\int_{0}^{1}\int_{0}^{1}F^{''}\left(\tau u_{1}+(1-\tau)u_{2}\right)\left(r,\tau hdt\right)d\tau \bar{r}dx \\
\leq& -\mathcal{R}\int_{0}^{1}\int_{0}^{1}F^{''}\left(\tau u_{1}+(1-\tau)u_{2}\right)\left(r,\tau bdW_{1}\right)d\tau \bar{r}dx+C\left(1+\sum \limits_{i=1}^{2}H^{3\sigma+1}(u_{i})\right)\|r\|_{1}\|r\|_{\frac{3}{4}}dt.
\end{align*}
Similarly
\begin{align*}
&\Rmnum{2} \\
\leq& -\mathcal{R}\int_{0}^{1}\int_{0}^{1}F^{''}\left(\tau u_{1}+(1-\tau)u_{2}\right)\left(r, (1-\tau)bdW_{2}\right)d\tau \bar{r}dx+C\left(1+\sum \limits_{i=1}^{2}H^{3\sigma+1}(u_{i})\right)\|r\|_{1}\|r\|_{\frac{3}{4}}dt.
\end{align*}
Since $|e_{n}|_{\infty}=1$, we have
\[
\Rmnum{3}\leq CB_{0}\left(1+\sum \limits_{i=1}^{2}H^{2\sigma}(u_{i})\right)|r|_{2}^{2}dt.
\]
Note that we have no information on the law of the couple $\left( W_{1},W_{2}\right)$. Hence, we cannot compute $d\left\langle(W_{1},e_{p}),(W_{2},e_{q})\right\rangle$. However, we know that
\[
d\left|\left\langle\left(W_{1},e_{p}\right),\left(W_{2},e_{q}\right)\right\rangle\right|\leq dt.
\]
It follows from Schwartz's inequality that
\[
\left(\sum \limits_{n=1}^{\infty}b_{n}\right)^{2}\leq \left(\sum \limits_{n=1}^{\infty}\mu_{n}b_{n}^{2}\right)\left(\sum \limits_{n=1}^{\infty}\frac{1}{\mu_{n}}\right)\leq CB_{1}.
\]
Hence
\[
\Rmnum{4}\leq CB_{1}\left(1+\sum \limits_{i=1}^{2}H^{2\sigma}(u_{i})\right)|r|_{2}^{2}dt.
\]
Likewise,
\begin{align*}
\Rmnum{5}\leq CB_{1}\left(1+\sum \limits_{i=1}^{2}H^{2\sigma}(u_{i})\right)|r|_{2}^{2}dt, \quad
\Rmnum{6}\leq CB_{0}\left(1+\sum \limits_{i=1}^{2}H^{2\sigma}(u_{i})\right)|r|_{2}^{2}dt.
\end{align*}
We combine these estimates to compute
\begin{align}\label{j1}
d&J_{*}(u_{1},u_{2},r)+2\alpha J_{*}(u_{1},u_{2},r)dt \nonumber \\
\leq& C\left(1+\sum \limits_{i=1}^{2}H^{3\sigma+1}(u_{i})\right)\|r\|_{1}\|r\|_{\frac{3}{4}}dt
-\mathcal{R}\int_{0}^{1}\int_{0}^{1}F^{''}\left(\tau u_{1}+(1-\tau)u_{2}\right)\left(r,\tau bdW_{1}\right)d\tau \bar{r}dx \\ \nonumber
&-\mathcal{R}\int_{0}^{1}\int_{0}^{1}F^{''}\left(\tau u_{1}+(1-\tau)u_{2}\right)\left(r,(1-\tau) bdW_{2}\right)d\tau \bar{r}dx.
\end{align}
Combining (\ref{z5}) and (\ref{j1}) leads us to the estimate
\begin{align*}
d&J(u_{1},u_{2},r)+2\alpha J(u_{1},u_{2},r)dt \\
\leq& C\left(1+\sum \limits_{i=1}^{2}H^{3\sigma+1}(u_{i})\right)\|r\|_{1}\|r\|_{\frac{3}{4}}dt
-\mathcal{R}\int_{0}^{1}\int_{0}^{1}F^{''}\left(\tau u_{1}+(1-\tau)u_{2}\right)\left(r,\tau bdW_{1}\right)d\tau \bar{r}dx \\
&-\mathcal{R}\int_{0}^{1}\int_{0}^{1}F^{''}\left(\tau u_{1}+(1-\tau)u_{2}\right)\left(r,(1-\tau) bdW_{2}\right)d\tau \bar{r}dx+\sum \limits_{i=1}^{2}\sigma G_{1} H^{\sigma-1}(u_{i})\left(M^{i},bdW_{i}\right)|r|_{2}^{2}.
\end{align*}
Since $\|r\|_{\frac{3}{4}}\leq CN^{-\frac{1}{4}}\|r\|_{1}$, there exists $\Lambda>0$ such that
\begin{align}\label{z7}
d&J(u_{1},u_{2},r)+\left(2\alpha-\frac{\Lambda}{N^{\frac{1}{4}}}l\left(u_{1}(t),u_{2}(t)\right)\right)J\left(u_{1},u_{2},r\right)dt \\ \nonumber
\leq& -\mathcal{R}\int_{0}^{1}\int_{0}^{1}F^{''}\left(\tau u_{1}+(1-\tau)u_{2}\right)\left(r,\tau bdW_{1}\right)d\tau \bar{r}dx \\ \nonumber
&-\mathcal{R}\int_{0}^{1}\int_{0}^{1}F^{''}\left(\tau u_{1}+(1-\tau)u_{2}\right)\left(r,(1-\tau) bdW_{2}\right)d\tau \bar{r}dx+\sum \limits_{i=1}^{2}\sigma G_{1} H^{\sigma-1}(u_{i})\left(M^{i},bdW_{i}\right)|r|_{2}^{2} \\ \nonumber
:=&dM(t).
\end{align}
Multiplying (\ref{z7}) by $\exp\left(2\alpha t-\frac{\Lambda}{N^{\frac{1}{4}}}\int_{0}^{t}l\left(u_{1}(s),u_{2}(s)\right)ds\right)$ and integrating from $0$ to $t\wedge\tau$, we see
\[
J_{FP}^{N}(u_{1},u_{2},r)(t\wedge \tau)\leq J_{FP}^{N}(u_{1},u_{2},r)(0)+\int_{0}^{t\wedge\tau}\exp\left(2\alpha s-\frac{\Lambda}{N^{\frac{1}{4}}}\int_{0}^{s}l\left(u_{1}(s^{'}),u_{2}(s^{'})\right)ds^{'}\right)dM(s).
\]
Take the expectation to conclude
\[
\mathbb{E}\left[J_{FP}^{N}(u_{1},u_{2},r)(t\wedge \tau)\right]\leq J\left(u_{0}^{1},u_{0}^{2},r_{0}\right).
\]
\\
\textbf{The second case:} $\mathbf{\lambda=-1}$.

It is not difficult to see that \[
J(u_{1},u_{2},r)=|\nabla r|_{2}^{2}+\mathcal{R}\int_{0}^{1}\int_{0}^{1}F^{'}\left(\tau u_{1}+(1-\tau)u_{2}\right)rd\tau\bar{r}dx\geq|\nabla r|_{2}^{2}.
\]
Utilizing It\^o's formula to $J(u_{1},u_{2},r)$, we see
\begin{align*}
d&J(u_{1},u_{2},r) \\
=& -2\alpha J(u_{1},u_{2},r) \\
&+\mathcal{R}\int_{0}^{1}\int_{0}^{1}F^{''}\left(\tau u_{1}+(1-\tau)u_{2}\right)\tau \left(r,-\alpha u_{1}dt-iAu_{1}dt-i|u_{1}|^{2\sigma}u_{1}dt+bdW_{1}+hdt\right)d\tau \bar{r}dx \\
&+\mathcal{R}\int_{0}^{1}\int_{0}^{1}F^{''}\left(\tau u_{1}+(1-\tau)u_{2}\right)\left(1-\tau\right) \left(r,-\alpha u_{2}dt-iAu_{2}dt-i|u_{2}|^{2\sigma}u_{2}dt+bdW_{2}\right)d\tau \bar{r}dx \\
&+\frac{1}{2}\sum\limits_{p=1}^{\infty}b_{p}^{2} \mathcal{R}\int_{0}^{1}\int_{0}^{1}F^{'''}\left(\tau u_{1}+(1-\tau)u_{2}\right)\left(r,\tau e_{p},\tau e_{p}\right)d\tau \bar{r}dx dt\\
&+\frac{1}{2}\sum\limits_{p,q=1}^{\infty}b_{p}b_{q} \mathcal{R}\int_{0}^{1}\int_{0}^{1}F^{'''}\left(\tau u_{1}+(1-\tau)u_{2}\right)\left(r,\tau e_{p},(1-\tau)e_{q}\right)d\tau \bar{r}dxd\left\langle\left(W_{1},e_{p}\right),\left(W_{2},e_{q}\right)\right\rangle \\
&+\frac{1}{2}\sum\limits_{p,q=1}^{\infty}b_{p}b_{q} \mathcal{R}\int_{0}^{1}\int_{0}^{1}F^{'''}\left(\tau u_{1}+(1-\tau)u_{2}\right)\left(r,(1-\tau) e_{q},\tau e_{p}\right)d\tau \bar{r}dxd\left\langle\left(W_{2},e_{q}\right),\left(W_{1},e_{p}\right)\right\rangle \\
&+\frac{1}{2}\sum\limits_{q=1}^{\infty}b_{q}^{2} \mathcal{R}\int_{0}^{1}\int_{0}^{1}F^{'''}\left(\tau u_{1}+(1-\tau)u_{2}\right)\left(r,(1-\tau) e_{q},(1-\tau) e_{q}\right)d\tau \bar{r}dxdt \\
:=&-2\alpha J(u_{1},u_{2},r)dt+\Rmnum{1}+\Rmnum{2}+\Rmnum{3}+\Rmnum{4}+\Rmnum{5}+\Rmnum{6}.
\end{align*}
As in our proof of the first case, we similarly conclude
\[
\mathbb{E}\left[J_{FP}^{N}(u_{1},u_{2},r)(t\wedge \tau)\right]\leq J\left(u_{0}^{1},u_{0}^{2},r_{0}\right).
\]
\end{proof}

\begin{cor}\label{corollary3.7}
Assume that for any $B$, $d_{0}$, $\kappa_{0}>0$, there exist $N_{1}(B, \kappa_{0})$ and $C^{*}(d_{0})$ such that under the assumptions of Lemma \ref{proposition3.6}, (\ref{z8}) and (\ref{z9}) hold with $N\geq N_{1}$, and for some $\rho> 0$,
\begin{align}\label{z10}
E_{u_{i},3\sigma+1}(t)\leq \rho +1+d_{0}^{3\sigma+1}+d_{0}^{6\sigma+2}+Bt \quad \quad on \ [0,\tau] \quad \quad for \ i=1,2,
\end{align}
then for any $u_{0}^{1}$, $u_{0}^{2}$ such that $\sum_{i=1}^{2}H(u_{0}^{i})\leq d_{0}$ and for any $a\in \mathbb{R}$,
\begin{align*}
\mathbb{P}\left(\|r(T)\|_{1}>C^{*}\left(d_{0}\right)\exp\left(a-\frac{\alpha}{4}T+\rho\right) \ and \ T\leq\tau\right)\leq \exp\left(-a-\frac{\alpha}{4}T\right).
\end{align*}
Furthermore, there exists a constant $C>0$ such that
\[
C^{*}(d_{0})\leq Cd_{0}\exp\left(Cd_{0}^{6\sigma+2}\right).
\]
\end{cor}

\begin{proof}
By Lemma \ref{proposition3.6} and Chebyshev's inequality, Corollary \ref{corollary3.7} can be verified.
\end{proof}

\begin{lem}\label{lemma3.8}
Assume that for any $B$, $d_{0}$, $\kappa_{0}>0$ and any $a\in\mathds{R}$, there exist $N_{2}(B, \kappa_{0}, a)$ and $C^{**}(d_{0},B)$ such that under the assumptions of Lemma \ref{proposition3.6}, (\ref{z8}) and (\ref{z9}) hold with $N\geq N_{2}$ and (\ref{z10}) holds for some $\rho> 0$, we obtain that for any T,
\[
\mathbb{P}\left(\int_{T}^{\tau}l\left(u_{1}(s),u_{2}(s)\right)\|r(s)\|_{1}^{2}ds>C^{**}(d_{0},B)\exp\left(a-\frac{\alpha}{2}T+\rho\right) \ and \ T\leq\tau\right)\leq \exp\left(-a-\frac{\alpha}{2}T\right),
\]
provided $\sum_{i=1}^{2}H(u_{0}^{i})\leq d_{0}$ holds. Furthermore, there exists a constant $C>0$ such that
\begin{align}\label{b2}
C^{**}(d_{0},B)\leq C(B)d_{0}\exp\left(Cd_{0}^{6\sigma+2}\right).
\end{align}
\end{lem}

\begin{proof}
Integrate (\ref{b1}) with respect to t:
\begin{align*}
\int_{T}^{\tau}J(u_{0}^{1},u_{0}^{2},r_{0})dt&\geq\int_{T}^{\tau}\mathbb{E}[J_{FP}^{N}(u_{1},u_{2},r)(t)]dt \\
&\geq\frac{1}{2}\int_{T}^{\tau}\mathbb{E}[\exp(2\alpha t-\frac{\Lambda}{N^{\frac{1}{4}}}t-C\frac{\Lambda}{N^{\frac{1}{4}}}(\rho+1+d_{0}^{3\sigma+1}+d_{0}^{6\sigma+2}+Bt))|\nabla r|_{2}^{2}]dt.
\end{align*}
So \\
\begin{align*}
\int_{T}^{\tau}\mathbb{E}[|\nabla r|_{2}^{2}dt]\leq 2\exp(-2\alpha T+\frac{\Lambda}{N^{\frac{1}{4}}}\tau+C\frac{\Lambda}{N^{\frac{1}{4}}}(\rho+1+d_{0}^{3\sigma+1}+d_{0}^{6\sigma+2}+B\tau))J(u_{0}^{1},u_{0}^{2},r_{0})(\tau-T).
\end{align*}
Since for any $x>0$, $1+x\leq C_{\delta}\exp{(\delta x)}$, we have
\begin{align*}
\mathbb{P}&(\int_{T}^{\tau}l(u_{1}(s),u_{2}(s))|\nabla r(s)|_{2}^{2}ds>C^{**}(d_{0},B)\exp(a-\frac{\alpha}{2}T+\rho)) \\
&\leq \frac{\mathbb{E}\int_{T}^{\tau}l(u_{1}(s),u_{2}(s))|\nabla r(s)|_{2}^{2}ds}{C^{**}(d_{0},B)\exp(a-\frac{\alpha}{2}T+\rho)} \\
&\leq \frac{\mathbb{E}\int_{T}^{\tau}C\exp{(\sum\limits_{i=1}^{2}E_{u_{i},3\sigma+1})}|\nabla r(s)|_{2}^{2}ds}{C^{**}(d_{0},B)\exp(a-\frac{\alpha}{2}T+\rho)} \\
&\leq \exp\left(-a-\frac{\alpha}{2}T\right),
\end{align*}
where we choose $d_{0}$ so small that there exist $N_{2}(B, \kappa_{0}, a)$ and $C^{**}(d_{0},B)$ such that the last inequality holds. Since
\[
J(u_{0}^{1},u_{0}^{2},r_{0})=|\nabla r(0)|_{2}^{2}+\mathcal{R}\int_{0}^{1}\int_{0}^{1}F^{'}(\tau u_{1}(0)+(1-\tau)u_{2}(0))r(0)d\tau \bar{r}(0)dx\leq Cd_{0},
\]
(\ref{b2}) follows.
\end{proof}

\subsection{Proof of Theorem \ref{a8}}

We choose $N_{0}=\max\left(N_{1}, N_{2}\right)$ and $N_{*}\geq N_{0}$.

\begin{proof}
Since the proof is similar to that of Theorem 2.10 in \cite{DO}, we outline the proof in this section for completeness. In order to prove this theorem, we will construct a coupling. The construction of the coupling is by induction. We first set
$u_{i}(0)=u_{0}^{i}, \ W_{i}(0)=0, \ i=1,2.$
Assuming that we have built $\left(u_{i},W_{i}\right)_{i=1,2}$ on $[0,kT]$, then we construct a probability space $(\Omega_{0},\mathcal{F}_{0},\mathbb{P}_{0})$ and two pairs of functions $\left(V_{i}^{a}\right)_{i=1,2}$ and $\left(V_{i}^{b}\right)_{i=1,2}$ satisfying (H3)-(H4) and independent of $\left(u_{i},W_{i}\right)_{i=1,2}$ on $[0,kT]$ and set for any $t\in [0,T]$, $i=1,2$,
\begin{equation*}
\left(u_{i}(kT+t),W_{i}(kT+t)\right)=
\begin{cases}
V_{i}^{a}\left(t,u_{1}(kT),u_{2}(kT)\right) \quad {\rm if} \ l_{0}(k)=\infty \ {\rm and} \ H\left(u_{1}(kT)\right)+H\left(u_{2}(kT)\right)\leq R_{0}, \\
V_{i}^{b}\left(t,u_{1}(kT),u_{2}(kT)\right) \quad {\rm if} \ l_{0}(k)\leq k, \\
V_{i}^{0}\left(t,u_{1}(kT),u_{2}(kT)\right) \quad {\rm if} \ l_{0}(k)=\infty \ {\rm and} \ H(u_{1}(kT))+H(u_{2}(kT))> R_{0},
\end{cases}
\end{equation*}
where $V_{i}^{0}\left(t,u_{1}(kT),u_{2}(kT)\right)$ is a trivial coupling. We take a cylindrical Wiener process W independent of $(u_{i},W_{i})_{i=1,2}$ on $[0,kT]$ and set $V_{i}^{0}(t,u_{1}(kT),u_{2}(kT))=(u(t,u_{i}(kT)),W)$. Indeed, the choice of the coupling is not important in the three case.

It suffices to verify
\begin{align}\label{zhengming1}
\mathbb{P}\left(\left|u_{1}(t)-u_{2}(t)\right|>C(1+t)^{-q}\right)\leq C(1+t)^{-q}\left(1+H(u_{0}^{1})+H(u_{0}^{2})\right).
\end{align}
We suppose $\varphi$ is a Lipschitz and bounded function, then we have
\begin{align*}
\left|\mathbb{E}\varphi\left(u\left(t,u_{0}^{1}\right)\right)-\mathbb{E}\varphi\left(u\left(t,u_{0}^{2}\right)\right)\right|&=\left|\mathbb{E}\varphi(u_{1}(t))-\mathbb{E}\varphi(u_{2}(t))\right| \\
&\leq 2|\varphi|_{\infty}\mathbb{P}\left(|u_{1}(t)-u_{2}(t)|>C(1+t)^{-q}\right)+CL_{\varphi}\left(1+t\right)^{-q}.
\end{align*}
By (\ref{zhengming1}), we find
\begin{align*}
\left|\mathbb{E}\varphi\left(u\left(t,u_{0}^{1}\right)\right)-\mathbb{E}\varphi\left(u\left(t,u_{0}^{2}\right)\right)\right|\leq C\|\varphi\|_{L}\left(1+H\left(u_{0}^{1}\right)+H\left(u_{0}^{2}\right)\right)(1+t)^{-q},
\end{align*}
which implies that Theorem \ref{a8} holds. In order to prove (\ref{zhengming1}), we show that (H1)-(H5) are true. Specifically, (H1) can be easily proved by the definition of $l_{0}$, (H2) can be obtained by Lemma \ref{proposition3.6} and Corollary \ref{corollary3.7}, and (H5) is the so-called Lyapunov structure and follows from Lemma \ref{lemma3.2}. The proof of (H3)-(H4) are completely similar to that of (2.3)-(2.4) in [13], except changing the exponent $\sigma$. So we omit the proof.

Since (H1)-(H5) hold, completely similar to the proof in \cite[Section 3]{DO}, we can conclude the proof.
\end{proof}

\section*{Acknowledgements}

This work is supported by NSFC Grants 11871132, 11925102, and Dalian High-level Talent Innovation Program (Grant 2020RD09).


\begin{thebibliography}{xx}
\bibitem{BRZ2014}
V. Barbu, M. R\"{o}ckner, D. Zhang, Stochastic nonlinear Schr\"{o}dinger equations with linear multiplicative noise: rescaling approach.  {\it J. Nonlinear Sci.} {\bf
24} (2014), 383--409.

\bibitem{BRZ2016}
V. Barbu, M. R\"{o}ckner, D. Zhang, Stochastic nonlinear Schr\"{o}dinger equations. {\it Nonlinear Anal.}  {\bf
136} (2016), 168--194.

\bibitem{BD1999}
A. de Bouard and A. Debussche, A stochastic nonlinear Schr\"{o}dinger equation with multiplicative noise.  {\it Comm. Math. Phys.} {\bf
205} (1999), 161--181.

\bibitem{BD2003}
A. de Bouard and A. Debussche, The stochastic nonlinear Schr\"{o}dinger equation in $H^{1}$.  {\it Stochastic Anal. Appl.} {\bf
21} (2003), 97--126.

\bibitem{BFZ1}
Z. Brze\'{z}niak, B. Ferrario, M. Zanella, Invariant measures for a stochastic nonlinear and damped 2D Schr\"{o}dinger equation. arXiv: 2106.07043.

\bibitem{BFZ2}
Z. Brze\'{z}niak, B. Ferrario, M. Zanella, Ergodic results for the stochastic nonlinear Schr\"{o}dinger equation with large damping. arXiv: 2205.13364.

\bibitem{BHL}
Z. Brze\'{z}niak, F. Hornung, L. Weis, Martingale solutions for the stochastic nonlinear Schr\"{o}dinger equation in the energy space. {\it Probab. Theory Related Fields.}  {\bf 174} (2019), 1273--1338.

\bibitem{BHM}
Z. Brze\'{z}niak, F. Hornung, U. Manna, Weak martingale solutions for the stochastic nonlinear Schr\"{o}dinger equation driven by pure jump noise. {\it Stoch. Partial Differ. Equ. Annal. Comput.}  {\bf 8} (2020), 1--53.

\bibitem{BM}
Z. Brze\'{z}niak and A. Millet, On the stochastic Strichartz estimates and the stochastic nonlinear Schr\"{o}dinger equation on a compact Riemannian manifold. {\it A. Potential Anal.} {\bf 41} (2014), 269--315.

\bibitem{C}
T. Cazenave, \emph{Semilinear Schr\"{o}dinger Equations}. Courant Lecture Notes in Mathematics, 10. New York University, Courant Institute of Mathematical Sciences, New York; American Mathematical Society, Providence, RI, 2003. xiv+323 pp.

\bibitem{DZ1992}
G. Da Prato and J. Zabczyk, \emph{Stochastic Equations in Infinite Dimensions}. Cambridge University Press, 1992. xviii+454 pp.

\bibitem{DZ1996}
G. Da Prato and J. Zabczyk, \emph{Ergodicity for Infinite-Dimensional Systems}. Cambridge University Press, 1996. xii+339 pp.

\bibitem{DO}
A. Debussche and C. Odasso, Ergodicity for a weakly damped stochastic non-linear Schr\"{o}dinger equation.  {\it J. Evol. Equ.} {\bf
5} (2005), 317--356.

\bibitem{EKM}
I. Ekren, I. Kukavica, M. Ziane, Existence of invariant measures for the stochastic damped Schr\"{o}dinger equation. {\it Stoch. Partial Differ. Equ. Anal. Comput.} {\bf 5} (2017), 343--367.

\bibitem{GV}
J. Ginibre and G. Velo, On a class of nonlinear Schr\"{o}dinger equations. I. The Cauchy problem, general case. {\it J. Functional Analysis} {\bf 32} (1979), 1--32.

\bibitem{GO}
O. Goubet, Asymptotic smoothing effect for a weakly damped nonlinear Schr\"{o}dinger Equation in $T^{2}$. {\it  J. Differential Equations} {\bf 165} (2000), 96--122.

\bibitem{GL}
W. Grecksch and H. Lisei, Stochastic nonlinear equations of Schr\"{o}dinger type.  {\it Stoch. Anal. Appl.} {\bf 29} (2011), 631--653.

\bibitem{HM}
M. Hairer and J.C. Mattingly, Ergodicity of the 2D Navier-Stokes equations with degenerate stochastic forcing. { \it Ann. of Math.} (2) {\bf 164} (2006), 993--1032.

\bibitem{HNT}
N. Hayashi, K, Nakamitsu, M. Tsutsumi, On solutions of the initial value problem for the nonlinear Schr\"{o}dinger equations. { \it J. Funct. Anal.} {\bf 71} (1987), 218--245.

\bibitem{H2020}
F. Hornung, The stochastic nonlinear Schr\"{o}dinger equation in unbounded domains and non-compact manifolds. {\it NoDEA Nonlinear Differential Equations Appl.} {\bf 27} (2020), 46 pp.

\bibitem{H2018}
F. Hornung, The nonlinear stochastic Schr\"{o}dinger equation via stochastic Strichartz estimates. {\it J. Evol. Equ.}  {\bf
18} (2018), 1085--1114.

\bibitem{KT}
T. Kato, On nonlinear Schr\"{o}dinger equations. {\it Ann. Inst. H. Poincar\'{e} Phys. Th\'{e}or.} {\bf 46} (1987), 113--129.

\bibitem{KL2015}
D. Keller and H. Lisei, Variational solution of stochastic Schr\"{o}dinger equations with power-type nonlinearity. {\it Stoch. Anal. Appl.} {\bf 33} (2015), 653--672.

\bibitem{KL2016}
D. Keller and H. Lisei, A stochastic nonlinear Schr\"{o}dinger problem in variational formulation.  {\it NoDEA. Nonlinear. Differential Equations Appl.} {\bf 23} (2016), Art. 22,27 pp.

\bibitem{KJ}
J.U. Kim, Invariant measures for a stochastic nonlinear Schr\"{o}dinger equation. {\it Indiana Univ. Math. J.} {\bf 55} (2006), 687--717.

\bibitem{KNS}
S. Kuksin, V. Nersesyan, A. Shirikyan, Exponential mixing for a class of dissipative PDEs with bounded degenerate noise.  {\it Geom. Funct. Anal.} {\bf 30} (2020), 126--187.

\bibitem{KS2012}
S. Kuksin and A. Shirikyan, \emph{Mathematics of two-dimensional turbulence}. Cambridge Tracts in Mathematics, 194. Cambridge University Press, Cambridge, 2012. xvi+320 pp.

\bibitem{KS}
S. Kuksin and A. Shirikyan, Coupling approach to white-forced nonlinear PDEs.  {\it J. Math. Pures Appl.} (9) {\bf 81}  (2002), 567--602.

\bibitem{M}
J.C. Mattingly, Exponential convergence for the stochastically forced Navier-Stokes equations and other partially dissipative dynamics.  {\it Comm. Math. Phys.} {\bf 230} (2002), 421--462.

\bibitem{O}
C. Odasso, Ergodicity for the stochastic complex Ginzburg-Landau equations. {\it Ann. Inst. H. Poincar\'{e} Probab. Statist.} {\bf 42} (2006), 417--454.

\bibitem{DZ1992}
R. Temam, \emph{Infinite-dimensional dynamical systems in mechanics and physics}. Second edition. Applied Mathematical Sciences, 68. Springer-Verlag, New York, 1997. xxii+648 pp.

\bibitem{W}
M.I. Weinstein, Nonlinear Schr\"{o}dinger equations and sharp interpolation estimates. {\it Comm. Math. Phys.} {\bf 87} (1982/83), 567--576.
\end{thebibliography}
\end{document}